\documentclass[12pt,twoside]{amsart}
\usepackage{amssymb}

\usepackage{enumerate}
\usepackage[all]{xy}
\usepackage[hyperindex=true,plainpages=false,colorlinks=false,pdfpagelabels]{hyperref}

\usepackage{verbatim}
\usepackage{pdfpages}
\usepackage{graphicx, color}
\usepackage{subfigure}

\newtheorem{proposition}{\textbf{Proposition}}
\newtheorem{lemma}[proposition]{\textbf{Lemma}}

\newtheorem{theorem}[proposition]{\textbf{Theorem}}

\newtheorem{question}[proposition]{\textbf{Question}}

\theoremstyle{definition}
\newtheorem{definition}[proposition]{\textbf{Definition}}

\newtheorem{example}[proposition]{\textbf{Example}}
\newtheorem{remark}[proposition]{\textbf{Remark}}
\newtheorem{acknowledgements}{\textbf{Acknowledgements}}

\newcommand{\del}{\partial
}
\newcommand{\Lie}[1]{\operatorname{\textsl{#1}}}

\newcommand{\GL}{\Lie{GL}}

\newcommand{\SO}{\Lie{SO}}

\newcommand{\SL}{{\rm SL}}

\newcommand{\End}{{\rm End}}

\newcommand{\SU}{{\rm SU}}

\newcommand\C{{\mathbb C}}

\newcommand\Z{{\mathbb Z}}

\newcommand{\R}{{\mathbb R}}

\newcommand{\dbar}{{\bar\partial}}

\numberwithin{proposition}{section}
\numberwithin{equation}{section}

\title{Higher solutions of Hitchin's self-duality equations}

\author{Lynn Heller}
\author{Sebastian Heller}

\address{Institut f\"ur Differentialgeometrie\\
Welfengarten 1\\
30167 Hannover\\
Germany} 
\email{lynn.heller@math.uni-hannover.de}
\email{seb.heller@gmail.com}

\date{\today}

\begin{document}
\begin{abstract}
Solutions of Hitchin's self-duality equations corresponds to special real sections of the
Deligne-Hitchin moduli space -- twistor lines. A question posed by Simpson in 1997 asks whether all real sections give rise to global solutions of the self-duality equations. An affirmative answer would allow for complex analytic procedure to obtain all solutions of the self-duality equations. The purpose of this paper is to construct counter examples given by certain (branched) Willmore surfaces in $3$-space (with monodromy) via the generalised Whitham flow. Though these sections do not give rise to global solutions of the self-duality equations on the whole Riemann surface $M$, they induce solutions on an open and dense subset of it. This suggest a deeper connection between Willmore surfaces, i.e., rank $4$ harmonic maps theory, with the rank $2$ self-duality theory.
\end{abstract}
\maketitle
\vspace{-0.5cm}
\tableofcontents

\section*{Introduction}

\noindent
The starting point of our investigations are Hitchin's self-duality  equations on a compact Riemann surface  \cite{Hi1} \[F^\nabla=-[\Phi,\Phi^*];\;\;\;\bar\partial^\nabla\Phi=0,\]
 where $F^\nabla$ is the curvature of a special unitary connection $\nabla$ on a rank $2$ hermitian bundle $V$ over the Riemann surface $M$, and $\Phi$ 
 is a $(1,0)$-form 
  with values in the trace-free endomorphism bundle End$_0(V)$.
This is a 2-dimensional reduction of the self-dual Yang-Mills equations invariant under the (unitary) gauge group. Though it cannot be explicitly solved so far,  the moduli space $\mathcal M$ of solutions  possesses a very rich geometric structure.\\

Restricting to irreducible solutions Hitchin \cite{Hi1} showed that  $\mathcal M$ is a smooth manifold of dimension $12g-12$ for Riemann surfaces of genus $g \geq 2$. Moreover, these irreducible solutions are uniquely determined by their Higgs pair $(\bar\partial^\nabla,\Phi)$ up to unitary gauge transformations. From this perspective $\Phi$ is a holomorphic End$_0(V)$-valued $1$-form for the holomorphic vector bundle $(V,\bar\partial^\nabla),$ and the irreducibility of the solution translates to the stability of the Higgs pair: $\Phi$-invariant holomorphic  line subbundles of $V$ have strictly negative degree. Conversely, Hitchin \cite{Hi1} and Simpson \cite{Si0} have shown independently that every stable Higgs pair  gives rise to an irreducible solution of the self-duality equations. 
Therefore, there exist a $1:1$ correspondence between the moduli spaces of stable Higgs pairs and irreducible self-duality solutions -- the Hitchin-Kobayashi correspondence. By construction, the moduli space of stable Higgs bundles $(\bar \del^\nabla, \Phi)$ is a holomorphic symplectic manifold containing the cotangent bundle of the moduli space of stable holomorphic bundles as an open dense subset. Thus through the Hitchin-Kobayashi correspondence (the smooth part of) $\mathcal M$ inherits a complex structure $I$.
\\

\noindent
From another point of view it was observed that the connection $\nabla+\Phi+\Phi^*$ is flat. Donaldson \cite{Do}, using Eells and Sampson's \cite{ES} heat flow construction, showed that every irreducible flat $ \SL(2,\mathbb C)$-connection uniquely determines
a solution of the self-duality equations (up to gauge-equivalence). Since the moduli space of irreducible flat $\SL(2,\mathbb C)$-connections is again a holomorphic symplectic manifold, $\mathcal M$ naturally inherits a second complex structure $J$. Composing the two complex structures, a third complex structure $K$ is obtained rendering $\mathcal M$ into a hyper-K\"ahler manifold: the three complex structures anti-commute and are K\"ahler with respect to the same natural $L^2$-metric. \\

\noindent
The transition between the different pictures and thus the dependence of the different complex structures of $\mathcal M$ on each other is difficult, except in the case where the underlying Riemann surface is a torus. 
The construction of the Deligne-Hitchin moduli space  $\mathcal M_{DH} \rightarrow \mathbb C P^1$ \cite{Si, Si2} is an effort to interpolate between these pictures using a parameter $\lambda \in \C P^1,$ where the Higgs pair can be found at $\lambda =0$ and the flat connection $\nabla + \Phi + \Phi^*$ at $\lambda = 1$. It is constructed such that the so-called associated family of flat connections 
\begin{equation}\label{SD}
\lambda\in\mathbb C^*\longmapsto \nabla^\lambda:=\nabla+\lambda^{-1}\Phi+\lambda\Phi^*
\end{equation}
yields a holomorphic section of $\mathcal M_{DH} \rightarrow \mathbb C P^1$ of a particularly simple form: it satisfies a reality condition and gives rise to a so-called twistor line when identifying $\mathcal M_{DH} \rightarrow \mathbb C P^1$ with the twistor
space of $\mathcal M$, see \cite{Si}. A natural question, due to Simpson \cite{Si}, is whether all real (holomorphic) sections are twistor lines, i.e., whether they all correspond to solutions of Hitchin's self-duality equations. As noted by Simpson, an affirmative answer would allow, at least ``philosophically'', for a complex analytic procedure to obtain all solutions of the self-duality equations.
\\

 \begin{figure}[b]
  \centering
  \includegraphics[width=0.475\textwidth]{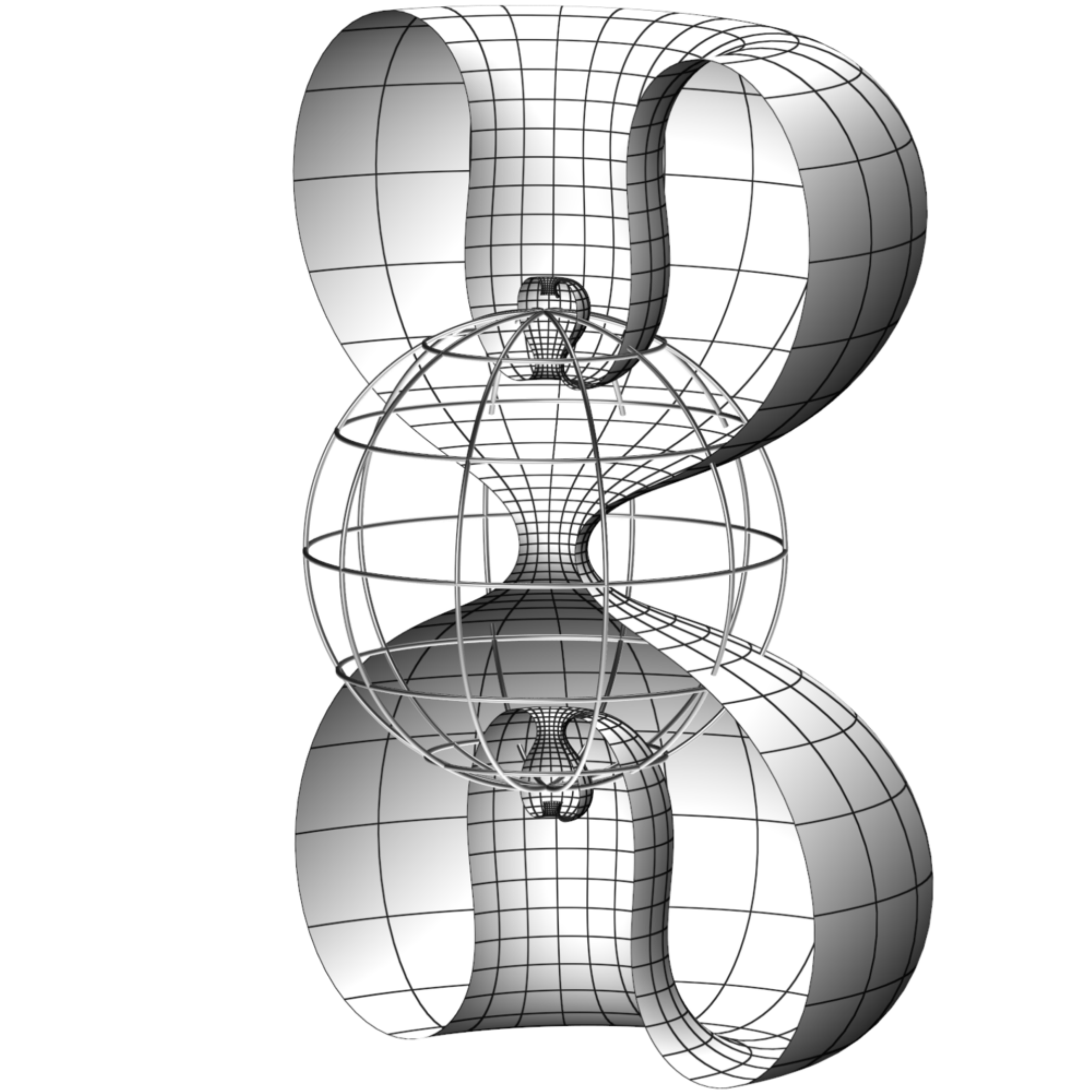}
  \caption{\small
The picture shows a Willmore cylinder in the round 3-sphere. The intersection of it with the two
   hyperbolic 3-spaces, given by the complement of a 2-sphere in $S^3$, is minimal.
    The surface shown is stereographically projected to $\R^3$,
    the wireframe designates the ideal boundary at infinity of the hyperbolic 3-space. Image by Nick Schmitt.}
  \label{figure1}
\end{figure}

This paper gives a negative answer to this question by constructing counter examples arising from certain Willmore surfaces. Willmore surfaces $f\colon M\to S^3$ are critical points of the Willmore functional \begin{equation}\label{Willmore}\int_M (|\vec{H}|^2 + 1) dA,
\end{equation} 
where $\vec{H}$ denotes the mean curvature of the immersion and $dA$ is the induced area form.
The Willmore functional is invariant under conformal transformations of the 3-sphere. Examples of Willmore surfaces
are given by minimal surfaces in the constant curvature subgeometries of the conformal 3-sphere.
Willmore surfaces 
have been studied via integrable systems techniques, see for example in \cite{FLPP,Bo}.
 We adjust the generalised Whitham flow for  minimal and constant mean curvature (CMC) surfaces in the $3$-sphere developed in \cite{HeHeSch} to flow from equivariant Willmore cylinders discovered by Babich and Bobenko \cite{BaBo} to (branched) Willmore surfaces of higher genus. The key observation here is that the Babich-Bobenko examples solve Hitchin's self-duality equations away from their umbilic lines, i.e., they are solutions to the self-duality equations away from 1-dimensional singularity sets on a torus, see Figure \ref{figure1}.

In order to construct real sections of the Deligne-Hitchin moduli space which do not correspond to twistor lines, we start with  Willmore surfaces of Babich-Bobenko type. We show that they correspond to  families of flat connections 
satisfying the reality condition of the self-duality equations.
Then we flow these initial families of flat connections with the generalised Whitham flow introduced in \cite{HeHeSch} towards families of flat connections on higher genus surfaces. At small rational times $\rho$ we obtain the desired counter examples on high genus surfaces. In order to avoid singular points of the moduli space, i.e., reducible flat connections, we drop the extrinsic closing condition of the surfaces and fix the spectral curve $\Sigma$ of the initial surface instead. When applied to solutions of spectral genus 0, the flow yields global ($\mathbb Z_{g+1}$-symmetric) solutions of the self-duality equations.  Therefore, we call these new real sections (corresponding to Willmore tori of spectral genus $1$) {\em higher solutions} of the self-duality equations.
 They turn out to solve the self-duality equations on an open and dense subset of the Riemann surface $M$. 

There are in fact two types of real sections covering the antipodal involution $\lambda\mapsto-\bar\lambda^{-1}$ of $\C P^1$ corresponding to the two real subgroups $\SU(2)$ and $\SU(1,1)$ of $\SL(2,\C)$. The $\SU(1,1)$-case corresponds to (equivariant) harmonic maps into the space of oriented circles in the 2-sphere $\SL(2,\C)/\SU(1,1)$, and examples 
are constructed in \cite{BHR}. In contrast to the examples constructed here, those differ from twistor lines by a  $\mathbb Z_2$-invariant, see Sections 
\ref{subSrealS} and \ref{relatedwork},
and are not related to solutions of the self-duality equations. \\

The paper is organised as follows. We first introduce the notion and the most important properties of the Deligne-Hitchin moduli space  $\mathcal M_{DH}$ of a compact Riemann surface in Section \ref{Sec1}. 
In Section \ref{sec2} we describe 
the families of flat $\SL(2,\C)$-connections on tori in terms of (algebro-geometric) spectral data . The spectral data (constructed in Theorem \ref{torus-spec-data}) will serve as initial data for a flow.
In section \ref{SecConstruction} we recall the construction of a 2:1 covering of the moduli space of
flat connections $\nabla^\lambda$ on the $4$-punctured sphere with prescribed local monodromies of $\nabla^\lambda$ around the punctures.  Thereafter, we use the eigenvalue $\rho$ of the logarithmic local monodromy as 
the flow parameter and adapt the generalised Whitham flow techniques of \cite{HeHeSch} to deform our  initial 
data on a torus in Section \ref{SecConstruction}. In Section \ref{sec:erns} 
we prove the existence of real sections of $\mathcal M_{DH}$ over 
the Riemann surface $M_q$  of genus $g(\rho)$  for rational $\rho=-\tfrac{p}{q}$ which map into  the smooth part of the moduli space. 
Finally, we show in Section \ref{sec5} that these new real sections give rise to solutions of the self-duality equations on open and dense subsets of the compact Riemann surface $M_q$.

\section{The Deligne-Hitchin moduli space $\mathcal M_{DH}$}\label{Sec1}

The Deligne-Hitchin moduli space $\mathcal M_{DH}=\mathcal M_{DH}(M)$ of a compact Riemann surface $M$
provides a natural tool to study associated families of flat connections of solutions to the self-duality equations. It was first defined by
Deligne (see \cite{Si, Si2}) as a complex analytic reincarnation of the twistor space (see \cite{HKLR}) associated to the hyper-K\"ahler moduli space  of self-duality solutions.

\begin{definition}
For $\lambda\in\C$ fixed,
a (integrable) $\lambda$-connection on a vector bundle $V\to M$ over a Riemann surface $M$ is a pair $(\dbar, D)$ consisting of a holomorphic structure on $V$ and a linear first order differential operator
\[D\colon\Gamma(M,V)\to\Omega^{(1,0)}(M,V)\]
satisfying the $\lambda$-Leibniz rule
\[D(fs)=\lambda\partial f\otimes s+f Ds\] for functions $f$ and sections $s$, and the integrability condition
\begin{equation}\label{intcond}
D\dbar+\dbar D=0.\end{equation}
\end{definition}
\begin{example}
Let $V = M \times \C^n$ be a trivial rank $n$ bundle equipped with the trivial connection $d.$ Let $d=d'+d''$
be its decomposition into its $(1,0)$ and $(0,1)$ parts.  Then, for $\lambda \in \C$ fixed, the pair
\[(\dbar_0,\lambda \del_0) = (d'',\lambda d') \]
gives the so-called trivial $\lambda$-connection on $V.$
 For $\lambda = 0$ the corresponding $\lambda$-connection reduces to the trivial holomorphic structure on $V.$
\end{example}

\begin{remark}
The operators  $D$ and $\dbar$ also act on $(0,1)$-forms and $(1,0)$-forms respectively. 
For $\lambda = 0$ the integrability condition \eqref{intcond} is equivalent to
 \[D=\Phi\in H^0(M, K \End(V)),\]
i.e., being complex linear and holomorphic, and for $\lambda \neq 0$ we have that 
\[\nabla=\tfrac{1}{\lambda}D+\dbar\]
is a flat connection. Here and in the following $K=(T^*M)^{(1,0)} $ is the canonical bundle of the Riemann surface $M.$
\end{remark}

\begin{definition}
A $\lambda$-connection $(\dbar, D)$ for $\lambda=0$ is called a Higgs pair. In this case,  $D=\Phi\in H^0(M, K\End(V))$ is tensorial and will be referred to as Higgs field. 
\end{definition}

In this paper we restrict to the subclass of $\lambda$-connections corresponding to the group $G_\C= \SL(2,\C).$ Note that a $\lambda$-connection on a vector bundle $V$ induces $\lambda$-connections on all associated tensor bundles, e.g. $V^*$ and $\Lambda^nV.$  For $\lambda=0$ and $n=\text{rank}(V)$, the induced $\lambda$-connection of $(\dbar, D)$ on  $\Lambda^nV$ is given
by the trace of $D$  and the induced holomorphic structure on the determinant bundle.
\begin{definition}
A $\SL (2, \C)$ $\lambda$-connection  is a $\lambda$-connection on a rank 2 vector bundle $V\rightarrow M$ over a compact Riemann surface $M$, such that
the induced $\lambda$-connection on $\Lambda^2 V$ is trivial.\end{definition}

For the rest of the section we consider the case where $M$ is compact and has genus $g \geq 2.$
Moreover, we assume without loss of generality that
\[V=M \times \C^2,\]
since every vector bundle $V$ with trivial determinant is (topologically) trivial. 
\begin{definition}
Let $M$ be a compact Riemann surface.
A $\SL(2, \C)$ $\lambda$-connection $(\dbar, D)$ is called stable, if
every $\dbar$-holomorphic subbundle $L\subset V=\underline\C^2$ with
\[D(\Gamma(M,L))\subset\Omega^{(1,0)}(M,L)\]
satisfies
\[\text{deg}(L)< 0\]

and semi-stable if 
\[\text{deg}(L)\leq 0.\]

All other $\lambda$-connections are called unstable. A $\SL(2,\C)$ $\lambda$-connection is called polystable if it is stable or the direct sum
of  $\lambda$-connections on degree $0$ line bundles.
\end{definition}

For $\lambda\neq0$, every $\lambda$-connection $(\dbar, D)$ is automatically semi-stable. Moreover, $(\dbar, D)$ is stable if and only if the connection $\nabla=\tfrac{1}{\lambda}D+\dbar$ is irreducible. For $\lambda=0$ there exist unstable $\lambda$-connections and their gauge orbits are infinitesimal close to the gauge orbits of (certain) stable $\lambda$-connections. Moreover, the gauge orbits of certain semi-stable $\lambda$-connections are infinitesimal close to each other, see also the notion of $\mathcal S$-equivalence for the case of holomorphic bundles \cite{NR}.
In order to obtain a well-behaved moduli space we restrict to polystable
$\lambda$-connections.

For $\lambda \in \C$ fixed, let $\mathcal A^2_\lambda$ denote  the space of (integrable) $\SL(2,\C)$  $\lambda$-connections, and $\mathring{\mathcal A}^2_\lambda$ the subspace of polystable $\lambda$-connections. Then there is a natural action of the gauge group
 \[\mathcal G=\{g\colon M\to \SL(2,\C)\}\]
 on $\mathcal A^2_\lambda$ and  for $\lambda\neq 0$ the quotient 
\[\mathring{\mathcal A}^2_\lambda/\mathcal G\]
 is biholomorphic to the moduli space of flat totally reducible $\SL(2,\C)$-connections.
 Recall that a totally reducible connection is by definition a direct sum of irreducible connections.
As such it is a a complex analytic space which is smooth away from (gauge orbits of) reducible flat connections. For $\lambda=0$ the quotient  
\[\mathring{\mathcal A}_0^2/\mathcal G\]
is the moduli space of polystable Higgs bundles $\mathcal M_{Dol}$.

\begin{definition}
Let $M$ be a compact Riemann surface of genus $g \geq 2$. The Hodge moduli space $\mathcal M_{Hod}=\mathcal M_{Hod}(M)$ is the space of all polystable, $\SL(2,\C)$ $\lambda$-connections on $M$ modulo gauge transformations.

The gauge-equivalence class of a $\lambda$-connection $(\lambda,\dbar,D)$
is denoted by $$[\lambda,\dbar,D]\in\mathcal M_{Hod}$$ or by
\[[\lambda,\dbar,D]_M\in\mathcal M_{Hod}( M)\]
to emphasis its dependence on the Riemann surface.
\end{definition}
\begin{remark}
The Hodge moduli space can be equipped with an algebraic structure through the GIT construction \cite{Si}.  We prefer to think of $\mathcal M_{Hod}$ as a complex analytic space (with quotient topology) 
whose smooth points are given by the gauge orbits of stable $\lambda$-connections. These form an open and dense subset in $\mathcal M_{Hod}$.  As a stable $\lambda$-connection does not permit non-trivial automorphisms (trivial automorphisms are constant multiples of the identity), the smooth structure can be constructed by  standard gauge theoretic methods. 
\end{remark}

The Hodge moduli space admits a holomorphic map 
\[f=f_M\colon\mathcal M_{Hod}\longrightarrow \C; \quad [\lambda,\dbar,D] \longmapsto\lambda\] whose fiber at
$\lambda=0$ is the (polystable) Higgs moduli space $\mathcal M_{Dol}$, and at $\lambda=1$ it is the deRham moduli space of flat (totally reducible) $\SL(2,\C)$-connections $\mathcal M_{dR}$, which we consider as complex analytic spaces endowed with their respective natural complex structures. 

\subsection{The gluing construction of the Deligne-Hitchin moduli space}
Let $M$ be a Riemann surface and $\overline M$ be its complex conjugate Riemann surface.
As differentiable manifolds we have $M\cong\overline M$ and thus their deRham moduli spaces of flat $\SL(2,\C)$-connections
are naturally isomorphic (as complex analytic spaces, not as algebraic spaces). Through the Deligne gluing \cite{Si2} 
\[\Psi\colon\mathcal M_{Hod}(M)\setminus f_M^{-1}(0)\to\mathcal M_{Hod}(\overline M)\setminus f_{\overline M}^{-1}(0);\;[\lambda,\dbar,D]_M\mapsto[\tfrac{1}{\lambda},\tfrac{1}{\lambda}D,\tfrac{1}{\lambda}\dbar]_{\overline M}\]
we can define the Deligne-Hitchin moduli space to be 
\[\mathcal M_{DH}=\mathcal M_{Hod}(M)\cup_\Psi\mathcal M_{Hod}(\overline M).\]

The Deligne-Hitchin moduli space admits a natural fibration  $f\colon\mathcal M_{DH}\to\C P^1$ whose
restriction to $\mathcal M_{Hod}(M)$ is $f_M$ and whose
restriction to $\mathcal M_{Hod}(\overline M)$ is $1/f_{\overline M}.$

\begin{remark}\label{smoothMdh}
Note that the Deligne-gluing  map $\Psi$ maps stable $\lambda$-connections on $M$ to stable $\tfrac{1}{\lambda}$-connections on $\overline M.$ Hence, it maps the smooth locus of $\mathcal M_{Hod}(M)$ (consisting of stable $\lambda$-connections) to the smooth locus of $\mathcal M_{Hod}(\overline M)$, and $\mathcal M_{DH}$ is equipped
with a structure of a complex manifold at all of its stable points.
\end{remark}

\begin{definition}
A section of $\mathcal M_{DH}$ is a holomorphic map 

$$s: \C P^1 \rightarrow \mathcal M_{DH}$$

such that $f \circ s = $Id. 
\end{definition}

It is well-known (and one of the motivation behind its definition is) that
\[f\colon \mathcal M_{DH}\longrightarrow \C P^1\] 
is holomorphic isomorphic to the twistor fibration $\mathcal P\to\C P^1$ of the hyper-K\"ahler metric on the moduli space of solutions to Hitchin's
self-duality equations, at least at the smooth points, see \cite{Si}. The isomorphism is given as follows.
Take a solution $(\nabla,\Phi)$ of the self-duality equation
and the twistor line
\[\lambda\longmapsto (\dbar^\nabla,\Phi,\lambda)\]
with respect to the $C^\infty$-trivialisation $\mathcal P\cong \mathcal M_{Dol}\times\C P^1.$
Then, this twistor line is holomorphically isomorphic to the section given by the holomorphic
map
\begin{equation}\label{twistorsection}
\lambda\in\C\longmapsto[\lambda,\dbar^\nabla+\lambda\Phi^*,\lambda \partial^\nabla+\Phi]_M\in\mathcal M_{Hod}(M)\subset\mathcal M_{DH}.\end{equation}

It follows from the work of Hitchin \cite{Hi1} and Donaldson \cite{Do} that every stable point
in $\mathcal M_{DH}$ uniquely determines a twistor line. 
\begin{definition}
A holomorphic section $s$ of $\mathcal M_{DH}$ is called stable, if the $\lambda$-connection
$s(\lambda)$ is stable for all $\lambda\in \C^*$  and if the Higgs pairs $s(0)$   on $M$ and $s(\infty)$ on $\overline{M}$ are stable.  
\end{definition}
Note that a twistor line $s$ is already stable if $s(\lambda_0)$ is stable for one $\lambda_0\in\C$.

\subsection{Automorphisms of the Deligne-Hitchin moduli space}
The Deligne-Hitchin moduli space admits some natural automorphisms which will play important roles in the later sections.
First of all, for every $\mu\in\mathbb C^*$ the (multiplicative) action of $\mu$ on $\C P^1$ has a natural lift
to $\mathcal M_{DH}$ by
\[\mu([\lambda,\dbar,D])=[\mu\lambda,\dbar,\mu D].\]

\begin{definition}
We denote by $N : \mathcal M_{DH} \rightarrow \mathcal M_{DH}$ 
the map given by 
\[[\lambda,\dbar,D] \longmapsto [-\lambda,\dbar, -D].\]
\end{definition}

Second, in the general case (e.g., for GL$(n, \C)$ rather than $\SL(2, \C)$ connections) taking the dual of 
a flat connection gives rise to an automorphism of the moduli space of flat connections which extends to an automorphism of
Deligne-Hitchin moduli space we denote by $\sigma$. 
The automorphism $\sigma\colon \mathcal M_{DH}\to\mathcal M_{DH}$ is given by
\[[\lambda,\dbar,D] \longmapsto [\lambda,\dbar^\star, D^\star],\]
where $()^\star$ denote the dual operator.
Since $SL(2,\C)$-connections and $\lambda$-connections are self-dual, $D$ is just the identity map in our case.

The last automorphism we introduce is anti-holomorphic and denoted by $C.$ 
\begin{definition}\label{defC2}
Let $C\colon \mathcal M_{DH}\longrightarrow \mathcal M_{DH}$ be the continuation of the map
 $$\tilde C: \mathcal M_{Hod}(M) \longrightarrow \mathcal M_{Hod}(\overline M)$$ given by 

\[\tilde C([\lambda,\dbar,D]_M) \longmapsto [\bar\lambda,\bar\dbar,\bar D]_{\overline M}.\]

To be more concrete, for  
\[\dbar = \dbar^0+\eta \quad \text{and} \quad D = \lambda(\partial^0)+\omega\]
where $d=\dbar^0+\partial^0$ is the trivial connection, $\eta\in\Omega^{0,1}(M,\mathfrak{sl}(n,\C)),$ and $\omega\in\Omega^{1,0}(M,\mathfrak{sl}(n,\C)),$ we define
the complex conjugate on the trivial $\C^n$-bundle over $\overline M$ to be
 \[\bar \dbar = \partial^0+\bar\eta \quad \text{and} \quad \bar D = \bar\lambda(\dbar^0)+\bar\omega.\]
\end{definition}

The map $C$ covers the map 
\[\lambda\in\C P^1\longmapsto \bar\lambda^{-1}\in\C P^1.\]
It is important to note that $C$ and $N$ commute. Moreover, both maps are involutive. Thus, their composition 
\[\mathcal T=CN\] is an involution as well,
covering the fixed-point free involution
$\lambda\mapsto - \bar\lambda^{-1}$ on $\C P^1.$

\begin{remark}\label{Remreal}
For Deligne-Hitchin moduli spaces associated to other Lie groups than $\SL(2,\C)$ there exist different anti-holomorphic involutions 
$$\mathcal M_{DH}\longrightarrow \mathcal M_{DH}$$
covering $\lambda\mapsto\bar\lambda^{-1}$.
For $\SL(n,\C)$ and $\GL(n,\C)$ , another natural choice 
is given by
\[[\lambda,\dbar,D]_M \longmapsto [\bar\lambda,\dbar^*, D^*]_{\overline M}\]
with respect to a hermitian metric on $V$. For $n=2$ we have
(with respect to the standard hermitian metric on $\underline{\C}^2\to M$)
\[\overline\dbar=\begin{pmatrix} 0&-1\\1&0\end{pmatrix} \; \dbar^*\begin{pmatrix}0&1\\-1&0\end{pmatrix}\;\; \text{ and }\;\; \overline D=\begin{pmatrix} 0&-1\\1&0\end{pmatrix}D^*\begin{pmatrix}0&1\\-1&0\end{pmatrix}.\]
 Thus both definitions coincide.

 In the case of rank 1 Deligne-Hitchin moduli spaces, we denote by $C$ the real involution induced by complex conjugation.
\end{remark}

\subsection{Real sections}\label{subSrealS}
In the following let $M$ be a compact Riemann surface of genus $g\geq2$.
We consider
the antiholomorphic involution of the associated Deligne-Hitchin moduli space 
\[\mathcal T=CN\colon\mathcal M_{DH}\longrightarrow\mathcal M_{DH} \]
covering \[\lambda\longmapsto-\bar\lambda^{-1}\] of $\mathbb CP^1.$
A holomorphic section  $s$ of $\mathcal M_{DH}$ is called real with respect to $\mathcal T$ if
\[\mathcal T(s(-\bar\lambda^{-1}))=s(\lambda)\]
holds for all $\lambda\in\mathbb CP^1$.
\begin{example}
Using Remark \ref{Remreal} we observe that twistor lines \eqref{twistorsection} are real holomorphic sections  with respect to $\mathcal T$
 for  $\mathcal G_\C=\SL(2,\C)$.  
\end{example}

In order to obtain a global lift of a section $s$ of $\mathcal M_{DH}$ to the space of flat connections or integrable $\lambda$-connections, it is technically necessary to consider $\lambda$-connections
with connection 1-forms that are only $\mathcal C^k$ on $M$ (rather than $\mathcal C^\infty$). For every $\lambda \in \C$ these $\mathcal C^k$-$\lambda$-connections are in fact gauge equivalent (by a gauge transformation of class $\mathcal C^{k+1}$) to  smooth $\lambda$-connections. We will make use of the following Lemma which is  proven analogously to the proof of Theorem 8 in \cite{H3}.
\begin{lemma}\label{lift-section}
Let $s$ be a holomorphic and stable section of $\mathcal M_{DH}\to\mathbb CP^1.$ Then there exists a holomorphic lift $\hat s_k$ of $s$ on $\mathbb C\subset\mathbb CP^1$  to the  space
of  $\mathcal C^k$-$\lambda$-connections on $M$  for every $k\in\mathbb N^{\geq2}$.
\end{lemma}

We restrict ourselves from now on to stable sections $s$. {\em Admissible} sections are particularly well-behaved sections $s$ of $\mathcal M_{DH}$ of the form
\[s(\lambda)=[\lambda,\bar\partial+\lambda \Psi,\lambda D+\Phi]\]
for a holomorphic structure $\bar\partial$, a $\partial$-operator $D$, an endomorphism-valued $(1,0)$-form $\Phi$
and an endomorphism-valued $(0,1)$-form $\Psi,$
such that $(\bar\partial,\Phi)$ and $(D,\Psi)$ are stable Higgs pairs on $M$ and $\overline M,$ respectively.
In particular, twistor lines are admissible.

Consider a real, holomorphic and stable section $s.$ By Lemma \ref{lift-section} $s$ admits a $\mathcal C^k$-lift
$$\lambda\mapsto (\lambda,\bar\partial^\lambda, D^\lambda).$$
Let
$$\lambda \in \mathbb C^*\subset\mathbb CP^1\mapsto \nabla^\lambda:=\bar\partial^\lambda+\tfrac{1}{\lambda}D^\lambda$$
 be the corresponding family of flat connections. Then $s$ being real translates to the existence of  a $\mathbb C^*$-family of $\SL(2,\C)$ gauge transformations $g(\lambda)$ satisfying
\begin{equation}\label{realeqsec}
\nabla^\lambda.g(\lambda)=\overline{\nabla^{-\bar\lambda^{-1}}}.
\end{equation}  
We call a family of flat connections $\nabla^\lambda$ satisfying \eqref{realeqsec} real.
Applying equation \eqref{realeqsec} twice we obtain 

$$\nabla^\lambda.g(\lambda) \overline{g(-\bar\lambda^{-1})} = \nabla^\lambda.$$

Because the section $s$ is stable, the connections $\nabla^\lambda$ are irreducible for all $\lambda \in \mathbb C^*$. Therefore $g(\lambda) \overline{g(-\bar\lambda^{-1})}$ is a constant multiple of the identity for every $\lambda \in \C^*$.
Moreover, we can assume the map $\lambda\in\C^*\mapsto g(\lambda)$ 
to be holomorphic in $\lambda$ by the constructions in the proof of Theorem 7 in \cite{H3}. Note that for every $\lambda\in\C^*$ the gauge 
$g(\lambda)$ has constant determinant $d(\lambda)$ on $M$, but the holomorphic map $\lambda\in\C^*\mapsto d(\lambda)$ cannot be chosen to be constant in general, see \cite[Proposition 2.11]{BHR}.
That we can find a family of SL$(2, \mathbb C)$ gauge transformations for real sections is the content of the next two lemmas.

\begin{lemma}\label{pmId}
Let $\lambda \in \C^* \mapsto \nabla^\lambda$ be a real and holomorphic family of irreducible flat $\SL(2,\C)$-connections and $g$ the corresponding holomorphic family of 
$\GL(2,\C)$-gauge transformations.
Then, there is a holomorphic map $h\colon U\to \C^*$ defined on an open neighbourhood $U$ of the closed unit disc $D_1:= \{ \lambda \in \C\; | \; |\lambda|^2 \leq 1\}$, such that $\tilde g:=hg$ satisfies
\begin{equation}\label{signagain}
\tilde g(\lambda)\overline{\tilde g(-\bar\lambda^{-1})}=\pm \text{Id}\end{equation} for all $\lambda\in\C^*$.
The sign on the right hand does not depend on the choice of $h$.
\end{lemma}
\begin{proof}
By irreducibility of $\nabla^\lambda$ we have
\[g(\lambda)\overline{ g(-\bar\lambda^{-1})}=\hat f(\lambda)\text{Id}\]
for a holomorphic function $\hat f$ without zeros along $S^1.$ Moreover, we can compute the index of the curve $\gamma = \hat f |_{S^1}$ to be
$$\text{Ind}_0(\gamma) = \tfrac{1}{2\pi i}\int_\gamma \tfrac{d\hat f}{\hat f}= \tfrac{1}{4\pi i}\int_\gamma \tfrac{d(\hat f^2)}{(\hat f)^2} = \tfrac{1}{4\pi i}\int_\gamma   \tfrac{ d \left(\det g (\lambda) \right )}{ \det g(\lambda) } +\tfrac{1}{4\pi i} \overline{ \int_\gamma  \tfrac{d \left(\det g (-\bar\lambda^{-1}) \right )}{ \det g (-\bar\lambda^{-1})}}  = 0.$$
Therefore, there exists a well-defined holomorphic function $f$ such that $\hat f(\lambda)=\exp({f(\lambda)} )$. Consider the Laurent series of $f$ for $\lambda \in S^1$ 
\[f(\lambda)=\sum_{k\in\mathbb Z} f_k\lambda^k.\]
Then $\hat f(\lambda)=\overline{\hat f(-\bar\lambda^{-1})}$ yields for $k\neq0$
\[(-1)^k\bar f_{k}=f_{-k}\]
and \[f_0=\bar f_0+n2\pi i,\]
for some $n\in\mathbb Z.$ Hence $\tilde g=h\, g$ for the holomorphic function
\[h(\lambda)=\exp\left ({-\sum_{k\in\mathbb Z^{>0}} f_k\lambda^k-\tfrac{1}{2} \text{Re}{f_0}} \right)\]
has the desired properties.

Let $\hat h$ be another holomorphic map such that $\hat g = \hat h \tilde g$ satisfies \eqref{signagain}. The map $\hat h$ is holomorphic and therefore we have
$$\hat h (\lambda) = \Sigma_{n = 0}^\infty a_n \lambda^n, \quad  \text{ and } \quad \overline{\hat h (-\bar\lambda^{-1})} = \sum_{n = 0}^\infty (-1)^n\bar a_n \lambda^{-n}.$$
Hence \eqref{signagain} yields 
$$ \hat h(\lambda) \overline{\hat h(-\bar \lambda^{-1})} = \pm 1.$$ 

Since $\hat h (\lambda)$ is positive and $\hat h(-\bar \lambda^{-1})$ is negative we have that both maps must be constant. Therefore 
$$ \hat h(\lambda) \overline{\hat h(-\bar \lambda^{-1})} = 1$$
showing the claimed independence of the choice of $h$ as long as $h$ is well-defined on the closed unit disc.
\end{proof}

Although not every section $s$ of $\mathcal M_{DH}$ admits lifts over $$\C\subset\C P^1 \quad \text{ and }\quad\C P^1\setminus\{0\}$$ which are  related by a $\SL(2,\C)$-valued family of gauge transformations over $\C^*$ (see Definition 2.8 and Proposition 2.11. in \cite{BHR}), this property holds for real  sections.

\begin{lemma}\label{lemma118}
Let $ \lambda \in \C^*\mapsto \nabla^\lambda$ be a real and holomorphic family of irreducible flat $\SL(2,\C)$-connections. Then there is a family of gauge transformations $g(\lambda)$ with det$[g(\lambda)] \equiv 1$ satisfying \eqref{realeqsec}.
 The family $g(\lambda)$ is unique up to sign.
\end{lemma}
\begin{proof}
Assume that $g$ cannot be chosen to be $\SL(2, \C)$-valued for all $\lambda\in\C^*$. Then the index of the curve $\tilde\gamma = $det$g|_{S^1}$ 
$$\text{Ind}_0(\tilde\gamma) = \tfrac{1}{2\pi i}\int_{\tilde \gamma} \tfrac{d \det(g)}{\det(g)}$$
must be odd. Otherwise the square root of $\det (g)$ would be well-defined  and $\tilde g= \tfrac{1}{\sqrt{\det(g)}} g$ defines a family of $\SL(2, \C)$-gauge transformations satisfying \eqref{realeqsec}. 
By multiplying $g$ with a  suitable holomorphic function defined on $\C^*$ the index of $\tilde\gamma$ changes by an even integer. Thus we can assume without loss of generality that \text{Ind}$_0(\tilde\gamma) =1$.  
Applying Lemma \ref{pmId} we can further assume 

\begin{equation}\label{posneg}
g(\lambda)\overline{g(-\bar\lambda^{-1})}=\pm\text{Id}.\end{equation}

For $p \in M$ fixed, consider the Birkhoff factorisation (see \cite[Chapter 8]{PS}  or Section \ref{sec5} for  a short summary) 
of $g_p(\lambda)$ 
\begin{equation}\label{bfac}
g_p(\lambda)=g_+(\lambda)\begin{pmatrix} \lambda^{k+1} & 0\\ 0 &\lambda^{-k}\end{pmatrix} g_-(\lambda)
\end{equation}
for some $k \in \mathbb N,$ where $g_+$ is a holomorphic map into $\SL(2,\C)$ that extends to $\lambda = 0$ while $g_-$ extends to $\lambda = \infty.$ The diagonal matrix diag$(\lambda^{k+1}, \lambda^{-k})$ accounts for the fact that \text{Ind}$_0(\tilde\gamma) =1$. By the uniqueness part of the Birkhoff factorisation every other pair $(\tilde g_+,\tilde g_-)$ satisfying \eqref{bfac} is given by 
\[\tilde g_+(\lambda)=g_+(\lambda)\begin{pmatrix} \tfrac{1}{a} & -\tfrac{1}{ad} b(\lambda^{-1})\lambda^{2k+1}\\ 0 & \tfrac{1}{d}\end{pmatrix}\]
and
\[\tilde g_-(\lambda)=\begin{pmatrix} a &  b(\lambda^{-1})\\ 0 & d\end{pmatrix} g_-(\lambda),\]
with constants $a,d\in\C^*$ and a polynomial $b$ (in the variable $\lambda^{-1}$) of degree at most $2k+1.$ 
By \eqref{posneg} we can relate the Birkhoff factorisations of $g(\lambda)$ and $\overline{g(-\bar\lambda^{-1})}$ :
\[g(\lambda)^{-1}=\pm\overline{g_+(-\bar\lambda^{-1})}\begin{pmatrix} (-1)^{k+1}\lambda^{-k-1} & 0\\ 0 &(-1)^{k}\lambda^{k}\end{pmatrix}\overline{g_-(-\bar\lambda^{-1})}\]
and therefore
\[g(\lambda)=\pm(-1)^k\overline{g_-(-\bar\lambda^{-1})}^{-1}\begin{pmatrix} -\lambda^{k+1} & 0\\ 0 &\lambda^{-k}\end{pmatrix}\overline{g_+(-\bar\lambda^{-1})}^{-1}.\]
Hence  there exist $a, d \in \C^*$ and a polynomial $b (\lambda^{-1})$ such that
\[\pm(-1)^k\overline{g_-(-\bar\lambda^{-1})}^{-1}\begin{pmatrix} -1 & 0\\ 0 &1\end{pmatrix}= g_+(\lambda)\begin{pmatrix} \tfrac{1}{a} & -\tfrac{1}{ad} b(\lambda^{-1})\lambda^{2k+1}\\ 0 & \tfrac{1}{d}\end{pmatrix}\]
and
\[\overline{g_+(-\bar\lambda^{-1})}^{-1}=\begin{pmatrix} a &  b(\lambda^{-1})\\ 0 & d\end{pmatrix} g_-(\lambda).\]
 Putting the last two equation together yields that either $-a\bar a=1$ or $-\bar d d=1$, depending on the sign of  $(-1)^k$ and the sign of \eqref{posneg},
 giving a contradiction in either case. The uniqueness of $g$ up to sign follows from the stability of $s$.
\end{proof}

For every $\C^*$-lift $\nabla^\lambda$ of a real holomorphic section $s$ of $\mathcal M_{DH}$  the two lemmas above yield the existence of  a holomorphic family of
 $\SL(2,\mathbb C)$-gauge transformations $g(\lambda)$ (unique up to sign)
  satisfying \eqref{realeqsec} and \begin{equation}\label{realeqsecsign}
g(\lambda)\overline{g(-\bar\lambda^{-1})}=\pm\text{Id}.\end{equation}
By \cite[Lemma 2.15]{BHR} the sign on the right hand side is independent of the lift $\nabla^\lambda$ of $s$  motivating the following definition.
\begin{definition}\cite[Definition 2.16]{BHR}
A stable real section $s$ of $\mathcal M_{DH}$ is called positive or negative depending on the sign of \eqref{realeqsecsign}.
\end{definition}

\begin{remark}
A real section $s$ corresponding to a solution of Hitchin's self-duality equations is negative.
In fact, a canonical lift is given by  the associated family of flat connections, and for the standard hermitian structure on $\mathbb C^2$ we obtain that
\[g(\lambda)=\begin{pmatrix}0&1\\-1&0\end{pmatrix}\] is constant in $\lambda$ and
squares to $-\text{Id}.$
\end{remark}

\begin{question}
Simpson \cite[$\S$4]{Si} raised the question whether every real holomorphic section of the Deligne-Simpson twistor space $\mathcal M_{DH}$ induces a solution of the self-duality equations.
\end{question}

\begin{remark}{Related work}\label{relatedwork} It is shown in \cite[Theorem 3.6]{BHR} 
that all stable, admissible, and negative sections of $\mathcal M_{DH}$ are twistor lines. 
In Simpson's notation \cite{Si3} being admissible is equivalent to being pure as a mixed twistor structure. The new sections of the Deligne-Hitchin moduli space constructed in this paper are stable and negative but not admissible.
\end{remark}

\subsection{The conformal Gauss map}\label{subsubdual}
In the definition of positive and negative real holomorphic sections, the stability of the lift is crucial. The conformal Gauss map is a geometric example where dropping the stability condition at $\lambda =0$ gives two lifts of the section $s$ on $\C^*$ with different signs in \eqref{realeqsecsign}.

Given a real holomorphic stable section $s$ of the Deligne-Hitchin moduli space with a (non-zero) nilpotent  Higgs field $\Phi$ at $\lambda=0$,
we consider the kernel bundle $$L:=\ker\Phi$$ and a complementary (smooth) subbundle $\tilde L$ of $V=\underline\C^2$.
Define the family of gauge transformations
\[\lambda\in\C^*\mapsto h(\lambda):=\begin{pmatrix} 1&0\\0&\lambda \end{pmatrix}\]
with respect to 
\[V=L\oplus \tilde L.\]
A direct computation gives (see \cite{BeHRo}) 
\begin{equation}\label{eq:dual}\tilde\nabla^\lambda:=\nabla^\lambda.h(\lambda)=\lambda^{-1}\tilde \Phi+\tilde\nabla+...\end{equation}
for a new nilpotent Higgs pair $(\dbar^{\tilde\nabla},\tilde\Phi).$ This Higgs pair $(\dbar^{\tilde\nabla},\tilde\Phi)$ is not stable,
since
$\tilde L=\ker\tilde\Phi$ and 
\[\deg(\tilde L)=-\deg(L)>0\] is positive
by assumption. Thus
$\tilde\nabla^\lambda$ is not a lift of $s$ on the whole complex plane $\C\subset\C P^1$.

 Lemma \ref{pmId} and Lemma \ref{lemma118} gives rise to a family of $\SL(2,\C)$ gauge transformations $ g(\lambda)$ satisfying
\[\overline{\nabla^{-\bar\lambda^{-1}}}=\nabla^\lambda. g(\lambda).\]
This family is unique up to sign. Moreover, the family of SL$(2, \mathbb C)$ gauge transformations
\begin{equation}\label{mcsgauge}
\tilde g(\lambda) :=\lambda h^{-1}(\lambda) g(\lambda)\overline{h(-\bar\lambda^{-1})}.\end{equation}
 satisfies
 \[ \overline{\tilde\nabla^{-\bar\lambda^{-1}}}=\tilde\nabla^\lambda.\tilde g(\lambda)\]

and a direct computation shows
\[\tilde g(\lambda)\overline{\tilde g(-\bar\lambda^{-1})}=-g(\lambda)\overline{g(-\bar\lambda^{-1})}.\]
This means that we have been able to change the sign by gauging to an unstable Higgs pair at $\lambda=0$.
\begin{remark}\label{rem:mcs}
The construction \eqref{eq:dual} is well-known in the theory of immersed surfaces in 3-space. In our case, a twistor line gives rise
to an equivariant harmonic map $f$ into the hyperbolic 3-space $\SL(2,\C)/SU(2).$
The Higgs field being nilpotent corresponds to 
the harmonic map being conformal, hence $f$ is minimal. 
Consider  the hyperbolic 3-space as a subspace of the round 3-sphere with a round 2-sphere as its boundary at infinity, e.g., the Poincare ball model. 
Every point of the surface uniquely determines  the best approximating 2-sphere at the point. In the minimal surface case, these 2-spheres are totally geodesic and  intersect the boundary at infinity  perpendicularly, see \cite{BuCa} and the references therein  or  \cite[$\S$ 4]{BeHRo} for a short summary. This yields a conformal harmonic map $G$ from $M$ into  the space of oriented circles in the 2-sphere (the boundary at infinity). The associated families of flat connections $\nabla^\lambda$ and $\tilde\nabla^\lambda$ of $f$ and $G,$ respectively, are related via the construction \eqref{eq:dual}. The map $G$ is called the conformal Gauss map or central sphere congruence.
\end{remark}
\begin{example}\label{rem:mcs2}
Even if the section $s$ is not admissible, i.e., the gauges $g(\lambda)$ do not admit a global Birkhoff factorisation into positive and negative gauges on $M$, $\tilde g(\lambda)$ in \eqref{mcsgauge} might. Let 
\begin{equation}\label{su11sd}
\lambda\in\C^*\mapsto\tilde\nabla^\lambda=\nabla +\lambda^{-1}\Phi+\lambda\Phi^{\#}\end{equation} be a family of flat connections, where $\nabla$ is a
$\SU(1,1)$-connection and $\#$ denotes the adjoint for the standard
$\SU(1,1)$ structure $(.,.)$ on $\C^2.$ We refer to $(\nabla,\Phi,(.,.))$ as a $\SU(1,1)$ self-duality solution. 

Assume there exist a curve $\gamma\subset M$ with
\[L= \ker(\Phi)=\ker(\Phi^{\#}),\]
i.e., $L$ is a null line with respect to $(.,.)$ along $\gamma.$ On $M \setminus \gamma$ we can reverse the construction in \eqref{eq:dual}
by taking  the gauge 
\[\tilde h (\lambda) = h^{-1}(\lambda)=\begin{pmatrix} 1&0\\0&\lambda^{-1}\end{pmatrix}\]
with respect to \[\underline\C^2=\ker(\Phi)\oplus \ker(\Phi^{\#}).\] It can be directly checked that the family
\[\tilde\nabla^\lambda.h^{-1}(\lambda)\]
solves the self-duality equations $M\setminus \gamma$ with respect to 
the hermitian metric  
\[(K.,.)\]
on $\underline\C^2$,
where 
\[K=\pm \begin{pmatrix}1&0\\0&-1\end{pmatrix}.\]
The sign depends on the component of $M\setminus \gamma$, and is chosen such that the unitary structure is positive definite.  Note that $h^{-1}$ becomes singular along $\gamma.$ 
The real holomorphic sections constructed in Section \ref{sec:erns} have the same behaviour as this example.
\end{example}

\section{Spectral genus 1 solutions on a torus}\label{sec2}
The aim of this section is to provide spectral data of singular solutions of the cosh-Gordon equation on a Riemann surface $M$ of genus 1. Away from the singularities these give rise to minimal surfaces in the hyperbolic 3-space and correspond therefore to local solutions of the self-duality equations. The spectral data will serve as initial data for the flow constructed in this paper.

We denote the rank one Deligne-Hitchin moduli space corresponding to the group $G_\C=\C^*$ of a Riemann surface  $M$ by $\mathcal M^1_{DH}(M)=\mathcal M^1_{DH}$. 
As in the $\SL(2,\C)$ case, the rank one Deligne-Hitchin moduli space of $M$ admits the automorphisms $N$ and the real involutions 
$C$ and $\mathcal T=CN.$
It also has an additional
holomorphic involution
\[\sigma\colon \mathcal M^1_{DH}(M)\to\mathcal M^1_{DH}(M); [\lambda,\bar\partial, D]\mapsto [\lambda,\bar\partial^\star, D^\star],\]
taking a $\lambda$-connection to its dual.
\begin{theorem}\label{torus-spec-data}
Let $\Gamma=2\Z+2\tau\Z$ and $\Lambda=\Z+\tau_{spec}\Z$ with $\tau\in i\R^{>0}$, $\tau_{spec}\in i\R^{>1}$ be rectangular lattices and let $M=\C/\Gamma$ and $\Sigma=\C/\Lambda$  be the corresponding Riemann surfaces. 
 Then there exists a holomorphic map
\[\mathcal D\colon\Sigma\to \mathcal M^1_{DH}(M)\] satisfying
\begin{enumerate}
\item $\lambda:=f_M\circ\mathcal D\colon\Sigma \to \C P^1$ is a double covering branched over $0,\infty,r,-\tfrac{1}{r}$
for some $0<r<1;$
\item there is a holomorphic involution $\sigma\colon\Sigma\to\Sigma$ such that
$$\mathcal D\circ\sigma=\sigma\circ\mathcal D \quad \text{ and } \quad\lambda\circ\sigma=\lambda;$$
\item there is a real involution $\eta\colon\Sigma\to\Sigma$ covering the antipodal involution\linebreak $\lambda\mapsto -\bar\lambda^{-1}$
such that
\[\sigma(\mathcal D(\eta(\xi)))=\mathcal T(\mathcal D(\xi))\]
for all $\xi\in\Sigma$.
\end{enumerate}
\end{theorem}

\begin{proof} On the Riemann surface $\Sigma$ 
consider the real fixed-point free involution
\[\eta\colon\Sigma\longrightarrow\Sigma;\quad [\xi] \longmapsto [\bar\xi+\tfrac{1+\tau_{spec}}{2}]\]
and the elliptic involution  
$$\sigma:[\xi]\longmapsto[-\xi].$$

Since $\eta$ commutes with $\sigma,$ it induces a real fixed-point free involution on $\Sigma/\sigma \cong \C P^1$ which, after applying a suitable Moebius transformation, is the antipodal map.
Thus there is a 2-fold covering 
\[\lambda\colon \Sigma\longrightarrow \C P^1\]

with $\lambda \circ \sigma = \lambda$,  $\lambda([0])=0,$ $\lambda\circ\eta(\xi)=-\overline{\lambda(\xi)^{-1}}$ for all $\xi\in\Sigma$. Without loss of generality we can assume
  \[r:=\lambda([\tfrac{1}{2}])\in\R\quad \text{ and } \quad R:=\lambda([\tfrac{\tau_{spec}}{2}])=-\tfrac{1}{r}\in\R.\]
The preimage of $\lambda = \infty$ is $\xi=[\tfrac{1+\tau_{spec}}{2}].$
It can be shown that for $\tau_{spec}\in i\R^{>1}$
\[0<r<1\quad \text{ or equivalently }\quad R<-1.\]

For fixed $\tau_{spec}$ and $\lambda: \Sigma \rightarrow \C P^1$
we choose constants $a,b\in \C$ such that the Weierstrass $\wp$-function on $\Sigma = \C /\Lambda$ satisfies
\begin{equation}\label{choiceab1}
\int_1 a\wp+bd\xi=2
\end{equation}
and
\begin{equation}\label{choiceab2}
\int_{\tau_{spec}} a\wp+bd\xi=0,
\end{equation}
 where we identify $\Lambda\cong H_1(\Sigma,\mathbb Z).$
In fact $a,b$ are real and given by
\[a=-\frac{\tau_{spec}}{\pi i}\quad \text{ and } \quad b=-2{\eta_3}{\pi i},\]
and $\eta_3=\zeta(\tfrac{\tau_{spec}}{2})$ for the Weierstrass $\zeta$-function, see \cite[page 10]{HeHeSch}.

Let $\chi$ to be the meromorphic function on $\C$ uniquely determined by
\begin{equation}
\label{def-chi}d\chi=\frac{\pi i}{2 \tau} (a\wp(\xi-\tfrac{1+\tau_{spec}}{2})+b)d\xi\;\;\text{ and }\;\; \chi(0)=0.\end{equation}
By construction, $\chi$ has a first order pole in $\xi = \tfrac{1+\tau_{spec}}{2}.$
Moreover, let $\alpha$ be the meromorphic function on $\C$ given by
\begin{equation}\label{def-alpha}\alpha(\xi)= \overline{\chi(\bar\xi-\tfrac{1+\tau_{spec}}{2})}+\frac{\pi i}{2 \tau}.\end{equation}
It has a first order pole in $\xi = 0.$

Consider the Riemann surface $M=\C/\Gamma$ equipped with its affine coordinate $w$ and  the map  
\[\mathcal D\colon\Sigma\to\mathcal M^1_{DH}(M);  \quad \xi \mapsto \left[\lambda(\xi), \bar\partial^0-\chi(\xi) d\bar w,\lambda(\xi) (\partial^0+\alpha(\xi)d w)\right]_M.\]

We want to show that $\mathcal D$ is well-defined on $\Sigma=\C/\Lambda$. 
The map $\mathcal D$ is well-defined on $$\C\setminus\left(\Lambda\cup\left(\tfrac{1+\tau_{spec}}{2}+\Lambda\right)\right).$$
It extends holomorphically to $\xi=0$ as $\alpha$ has a first order pole and $\lambda$ has a zero at $\xi=0$. Using the Deligne-gluing, 
$\mathcal D$ extends holomorphically to $\xi=\tfrac{1+\tau_{spec}}{2}$ as well.
Due to the periodicity of $\chi$ and $\alpha$,  $\xi$ and $\xi+\gamma$ define the same point in $\mathcal M^1_{DH}(M)$ for all $\gamma\in\Lambda$, showing that $\mathcal D$ is well-defined on $\Sigma$.  The properties (1)-(3) are easy to check.
\end{proof}
\begin{remark}
In the following, we refer to $\Sigma$ as the spectral curve and to $(\Sigma,\lambda,\chi,\alpha)$ as spectral data.
\end{remark}
\begin{lemma}\label{no-spin}
With the notations of Theorem \ref{torus-spec-data} and for
generic $\tau\in i\R^{>0}$  the image of the map
\[\xi\in \mathcal S:= \lambda^{-1}(S^1)\subset \Sigma\longmapsto [\dbar_0-\chi(\xi)d\bar w]\in \mathrm{Jac}(M)\]
 does not contain a spin bundle on $M$. 
In particular,  
 $$[\bar\partial-\chi d\bar w]\colon \lambda^{-1}(D_1)\subset\Sigma\longrightarrow \mathrm{Jac}(M)$$  only maps $[0],[\tfrac{1}{2}]\in\Sigma$ to the trivial line bundle in $\mathrm{Jac}(M)$  for small $\tau\in i\R^{>0}$.
 \end{lemma}
 \begin{remark}
A spin bundle $S$ on a Riemann surface $M$ is a holomorphic line bundle for which $S^2=K$ is the canonical bundle of $M.$ In the case
of $M$ being a torus, a spin bundle is a square root of the trivial holomorphic line bundle.
\end{remark}

\begin{proof}
The holomorphic structure 
\[\dbar_0-\chi d\bar w \quad \text{on}\quad \underline\C^2\to M\]
 is spin if and only if 
\begin{equation}\label{spinalong}\tfrac{2\tau}{\pi i}\chi\in \mathbb Z\oplus\tau\mathbb Z.\end{equation}
 We first show that $\tfrac{2\tau}{\pi i}\chi|_\mathcal S$ is not real valued, i.e., that its imaginary part never vanishes.  Note that for
 $\tau_{spec}\in i\R^{>1}$ one component of $\mathcal S$ is contained in  $$\{x+iy\mid 0<x<1;\; 0<y<\tfrac{\tau_{spec}}{2}\}.$$
Due the properties of the Weierstrass $\zeta$-function for rectangular lattices, the function $\tfrac{2\tau}{\pi i}\chi$ has real values along the real line and along the line $\tfrac{\tau_{spec}}{2}+\R$. Moreover,  
$$\tfrac{2\tau}{\pi i}\chi(\xi) \in i \R^{>0}\quad \text{ for } \quad \xi \in \{i y\mid 0<y<\tfrac{\tau_{spec}}{2}\}$$
 and 
$$ \tfrac{2\tau}{\pi i}\chi(\xi)\in 2+i \R^{>0}   \quad \text{ for } \quad \xi \in \{1+i y\mid 0<y<\tfrac{\tau_{spec}}{2}\}.$$ 
Let
  $$\mathcal R := \{\xi \in \{x+iy\mid 0<x<1;\; 0<y<\tfrac{\tau_{spec}}{2}\}\mid \tfrac{2\tau}{\pi i}\chi(\xi) \in  \R\}.$$
Since the function $\tfrac{2\tau}{\pi i}\chi$ has only simple poles and its critical points (which are all
of order one) are contained in $\mathbb Z+i\R$, we have 
$$\mathcal R \cap \del  \{x+iy\mid 0<x<1;\; 0<y<\tfrac{\tau_{spec}}{2}\} = \emptyset.$$
Therefore, $\mathcal R$ is a closed submanifold in 
$$\{x+iy\mid 0<x<1;\; 0<y<\tfrac{\tau_{spec}}{2}\}\subset \C.$$
If $\mathcal R$ would be non-empty, it contains a critical point of $\tfrac{2\tau}{\pi i}\chi$ giving a contradiction. 

By \eqref{def-chi} the map $\tfrac{2\tau}{\pi i}\chi$ does not depend on $\tau,$ its real part is real analytic and non-constant on $\mathcal S$. Since the imaginary part of $\tfrac{2\tau}{\pi i}\chi$ never vanishes, the set of $\tau \in i\R^{>0}$ for which \eqref{spinalong} holds can only be discrete. 
\end{proof}
\begin{lemma}\label{no-unit-spin}
With the notations of Theorem \ref{torus-spec-data} and for  $\tau\in i\R^{>0}$ small enough, the line bundle connection
\[d-\chi(\xi)d\bar w+\alpha(\xi) dw\] 
is never a non-trivial spin connection on the torus $M$ for all $\xi \in \lambda^{-1}(D_1).$
\end{lemma}
\begin{proof}

Using the reality condition \eqref{def-alpha} of the spectral data we only have to show this property  for $\xi$ lying in one of the two connected components of $\lambda^{-1}(D_1)\subset  \C/\Lambda  \cong \Sigma$. Without loss of generality we thus consider the connected component $\lambda^{-1}(D_1)$ containing the real axis in the following. The proof of Lemma \ref{no-spin} then gives a unique $\xi_0 \in \lambda^{-1}(D_1)$  such that $ \bar\partial-\chi(\xi_0)dw$ is a non-trivial spin structure. This $\xi_0$ is shown to be real and furthermore $\xi_0 \in [0, \tfrac{1}{2}]$.

Restricted to the real axis, the function $\chi$ is monotonic increasing, and the function $\alpha$ is monotonic decreasing.  
To be a spin connection, the connection
\[d-\chi(\xi)d\bar w+\alpha(\xi) dw\] 
must have $\pm 1$ monodromy. But since 
the connection is trivial for $\xi= \tfrac{1}{2}$, i.e., has monodromy $1$, the connection $d^{\chi(\xi_0),\alpha(\xi_0)}$  cannot be spin.
\end{proof}

\section{Deformations}\label{SecConstruction}
In this section we adjust the generalised Whitham flow in \cite{HeHeSch} to our spectral data $(\Sigma, \lambda, \chi, \alpha).$
 We start with  describing useful coordinate systems on the moduli spaces of regular singular connections on  a 4-punctured sphere. 

\subsection{Abelianization of flat connections}$\;$\\
We first recall the constructions of \cite{HeHe}, see also \cite[$\S 3.1$]{HeHeSch}.
Let $M=\C/\Gamma$ where $ \Gamma=2\mathbb Z+2\tau\mathbb Z$ with $\tau\in i\R^{>0}$ is a rectangular lattice. 
 Let $\sigma \colon M \to \C P^1$ be the elliptic involution $[w]\mapsto[-w]$ and let \begin{equation}\label{z}z\colon M\to\C P^1\end{equation}
 be the induced double covering which has four ramification points   \begin{equation}\label{rampe}P_1=[0],\,P_2=[1],\,P_3=[1+\tau],\,P_4=[\tau]\in M.\end{equation}
After a Moebius transformation 
the four branch points $p_k\in \C P^1$ of $z$
can be chosen without loss of generality to be
 \begin{equation}\label{brampe}
 p_1=z([0])=0,\, p_2=z([1])=1,\, p_3=z([1+\tau])=\infty,\, p_4=z([\tau])=m\end{equation}
 for some $m\in\C\setminus\{0,1\}.$
 For $\rho\in]-\tfrac{1}{2},\tfrac{1}{2}[$  consider the moduli space 
 \[\mathcal M^2_\rho(\C P^1\setminus\{p_1,\dots,p_4\})\] 
 of flat $\SL(2,\C)$ connections on the 4-punctured sphere $\C P^1\setminus\{p_1,\dots,p_4\}$ such that the local monodromies  around every puncture $p_k$ lie in  the conjugacy class of
\begin{equation}\label{local_monodromies}
\begin{pmatrix}
\exp{(2\pi i\frac{2\rho+1}{4})} & 0 \\ 
0 & \exp{(-2\pi i\frac{2\rho+1}{4})} \end{pmatrix}.
\end{equation}

For convenience of the reader  we shortly describe how $\mathcal M^2_\rho$ can be parametrized.
Consider the lattice $\tfrac{1}{2}\Gamma=\Z+\tau\Z$ and
the corresponding theta-function $\vartheta\colon\C\to\C$ of $\tfrac{1}{2}\Gamma$  uniquely determined (up to a multiplicative constant) by $\vartheta(0) = 0$ and
\begin{equation}\label{theta-function}
\vartheta(w+1) = \vartheta (w),\,\,  \vartheta(w+ \tau) = - \vartheta (w)e^{-2\pi i w}\end{equation}
for all $w\in\C.$
For $x \in\C \setminus\tfrac{1}{2}\Gamma$ fixed define 
\begin{equation}\label{beta-function}
\beta_{x}(w) = \frac{\vartheta(w-  x)}{\vartheta(w)}e^{\tfrac{2\pi i }{\bar\tau-\tau} x(w-\bar w)}.\end{equation}
The function $\beta_x$ is doubly periodic in $w$ with respect to the lattice  $\tfrac{1}{2}\Gamma$
and satisfies 
\[\left(\dbar-\frac{2\pi i}{\tau-\bar\tau}xd\bar w \right)\beta_{x}=0.\] Thus $\beta_x$ is a meromorphic section of the trivial bundle $\underline\C\to\C/\tfrac{1}{2}\Gamma$ equipped with the holomorphic structure $\dbar-\frac{2\pi i}{\tau-\bar\tau}xd\bar w.$  It has a simple zero at $w=x$ and a first order pole  at $w = 0.$  
By pull-back we can also consider $\beta_x$ on the fourfold cover $M = \C/ \Gamma\to \C/\tfrac{1}{2}\Gamma$  as a meromorphic section with simple poles at the half lattice points. \\

Let $\rho\in]-\tfrac{1}{2},\tfrac{1}{2}[.$
For a given flat $\C^*$-connection
\begin{equation}\label{linebundle_connection1form}
d^{\chi,\alpha}=d+\alpha dw-\chi d\bar w
\end{equation}
with $\chi\in\C\setminus(\frac{\pi i}{\tau-\bar\tau}\Z+\frac{\pi i\tau}{\tau-\bar\tau}\Z)$, 
and $\alpha\in\C$ let $x=\frac{\tau-\bar\tau}{2\pi i}\chi$  and define the flat singular connection $^\rho\hat\nabla^{\chi,\alpha}$
 on the trivial rank $2$ bundle
$\underline\C^2\to M$: 
 \begin{equation}\label{connection1form}
^\rho\hat\nabla^{\chi,\alpha}=d+ \begin{pmatrix}-\chi d\bar w + \alpha dw & \rho\frac{\vartheta'(0)}{\vartheta(-2x) }\beta_{2x}(w) dw \\ \rho\frac{\vartheta'(0)}{\vartheta(2x)}\beta_{-2x}(w) dw & \chi d\bar w -\alpha  d w  \end{pmatrix}.
 \end{equation}
 For $\rho=0$ we obtain smooth and totally reducible $\SL(2,\C)$-connections on the torus $M.$
The off-diagonal part of the connection 1-from in \eqref{connection1form} only depends on $\rho$ and on the holomorphic structure $\dbar-\chi d\bar w$ and is independent of $\alpha.$ By \cite[$\S3$]{HeHe}   $^\rho\hat\nabla^{\chi,\alpha}$ is  gauge equivalent 
(via a  gauge transformation with singularities at $P_1,\dots,P_4$) to an invariant connection  $^\rho\tilde\nabla^{\chi,\alpha}$ with respect to $[w]\mapsto[-w]$. In other words, $^\rho\tilde\nabla^{\chi,\alpha}$ is well defined on $\C P^1\setminus\{p_1,\dots,p_4\})$. 
This determines a map 
\begin{equation}\label{DefPi}\Pi\colon \mathcal M_{dR}^1(M)\setminus(\pi^1)^{-1}(\Lambda)\to \mathcal M^2_\rho(\C P^1\setminus\{p_1,\dots,p_4\});\quad 
[d^{\chi,\alpha}]\mapsto[^\rho\tilde\nabla^{\chi,\alpha}],\end{equation}
where
$$\pi^1\colon \mathcal M_{dR}^1(M)\to \mathrm{Jac}(M)$$ 
is the natural projection.
Replacing 
$d^{\chi,\alpha}$
by its dual connection $d^{-\chi,-\alpha}$ in \eqref{connection1form}
the corresponding connection lies in the same gauge equivalence class, i.e., \[\Pi([d^{-\chi,-\alpha}])=\Pi([d^{\chi,\alpha}])=[^\rho\tilde\nabla^{\chi,\alpha}]\]
giving the first assertion of the following theorem.

\begin{theorem}[\cite{HeHe} Theorem 3.4 and Theorem 3.5]\label{2:1} Let  $\rho\in]-\tfrac{1}{2},\tfrac{1}{2}[$. Then  the map  $\Pi$ in \eqref{DefPi}
is a double covering onto an open and dense subset.

 This 2:1 correspondence extends holomorphically to $\chi=\gamma\in\Lambda\equiv \frac{\pi i}{\tau-\bar\tau}\Z+\frac{\pi i\tau}{\tau-\bar\tau}\Z$ 
 along curves $\chi\mapsto(\chi, \alpha(\chi))$ 
 if and only if $\alpha$ expands around $\chi=\gamma$ as
\begin{equation}\label{a_spin_expansion}
\alpha(\chi)\sim_\gamma\pm\frac{4\pi i}{\tau-\bar\tau}\frac{\rho}{\chi-\gamma}+\bar\gamma+\,\text{ higher order terms in } (\chi-\gamma).\end{equation}
 \end{theorem}

\begin{remark}\label{Sign}
The ambiguity of the sign of the residue $\pm\frac{4\pi i \rho}{\tau-\bar\tau}$ in \eqref{a_spin_expansion} is meaningful: If $\rho>0$ the underlying parabolic structure is stable for the $``+"$-sign and unstable otherwise (see \cite[Theorem 3.5]{HeHe} for more details). This is reversed for  $\rho<0$. \end{remark}

In \cite{HeHeSch} the implicit function theorem is used to obtain a family of maps $\chi$ and $\alpha$ depending on $\rho$ satisfying closing conditions. There the most important condition was that the constructed family of flat connections $\nabla^\lambda$ is unitary along the unit circle. In this paper the considered real involution $\eta$ has no fixed points and thus we need a $2$-point construction to realise the reality condition.

\begin{lemma}\label{NSsec}
Let $\dbar_0-\chi_i d\bar z$, $i=1,2,$ be two non-spin holomorphic structures on the rectangular torus $M = \C /\Gamma$. 
Consider 
the flat $SL(2,\C)$-connections
\[^{\rho=0}\hat\nabla^{\chi_1,\bar\chi_2}\quad\text{ and }\quad ^{\rho=0}\hat\nabla^{\chi_2,\bar\chi_1}\]
given by \eqref{connection1form}.
Then, there is an open neighbourhood $U_\rho \subset \R$ of $\rho=0$, and open neighbourhoods $U_{\chi_i}$, $U_{\bar \chi_i} \subset \C$ of
$\chi_i$, and $\bar \chi_i$ respectively, with the property that for all $\rho\in U_\rho$,
$\tilde\chi_1\in U_{\chi_1}$ and $\tilde\chi_2\in U_{\chi_2}$ there are unique
\[\tilde\alpha_i=\tilde\alpha_i(\rho,\tilde \chi_1,\tilde \chi_2)\in U_{\bar \chi_i}\]
for $i=1,2$ such that
\[^\rho\hat\nabla^{\tilde\chi_1,\tilde\alpha_1}\;\;\;\text{ and }\;\;\;\; \overline{^\rho\hat\nabla^{\tilde\chi_2,\tilde\alpha_2}}\]
are gauge equivalent on $M\setminus\{P_1,..,P_4\}.$ Moreover, $\tilde\alpha_i=\tilde\alpha_i(\rho,\tilde \chi_1,\tilde \chi_2)$ is real analytic in its parameters.

\end{lemma}
The proof is a direct application of the implicit function theorem (at $\rho=0$) to Theorem \ref{2:1} , compare also with the proof of \cite[Lemma 3.2]{HeHeSch}.
\begin{remark}\label{NSsecRem}
The statement  of the lemma also holds
for the flat $\SL(2,\C)$-connections
\[^{\rho=0}\hat\nabla^{\chi_1,\bar\chi_2}\quad\text{ and }\quad ^{\rho=0}\hat\nabla^{\chi_2+\gamma,\bar\chi_1+\bar \gamma}\]
if $\gamma\in\Lambda\equiv \frac{\pi i}{\tau-\bar\tau}\Z+\frac{\pi i\tau}{\tau-\bar\tau}\Z$ is the lattice of Jac$(M).$
Therefore, the lemma can be rephrased in terms of the moduli space of flat line bundle connections on $M$.
\end{remark}

\subsection{The generalised Whitham flow}$\;$\\
For $\rho = 0$ consider spectral data $(\Sigma, \lambda, \chi, \alpha)$ of Theorem \ref{torus-spec-data} and the corresponding family of flat connections  via \eqref{connection1form}.  
The generalised Whitham flow is the deformation of spectral data (and of the resulting harmonic map) given by
varying $\rho$ while preserving the reality condition and the translational periods  \eqref{choiceab1} and \eqref{choiceab2} of $\chi$ and $\alpha$.
Hence, for some $d>1$, we are looking for deformations $\tilde \chi$ and $\tilde \alpha$ of $\chi $ and $\alpha$ of the form
\[\tilde\chi=\chi+\hat x\quad\text{ and }\quad \tilde\alpha=\alpha+\hat a,\]
where 
\[\hat x\colon \lambda^{-1}(\{\lambda\in\C\mid \lambda\bar\lambda<d\})\subset \Sigma\to \C\]
and
\[\hat a\colon \lambda^{-1}(\{\lambda\in\C\mid \tfrac{1}{d}<\lambda\bar\lambda<d\})\subset \Sigma\to \C\]
are odd (with respect to the involution $\sigma$ on $\Sigma$) holomorphic maps. In order to study these maps we first recall that the spectral curve $\Sigma$ in the torus case 
($\rho=0$, Theorem \ref{torus-spec-data})
is given
by the algebraic equation
\begin{equation}\label{ylambda}y^2=\lambda(\lambda-r)(\lambda+\tfrac{1}{r})\end{equation}
for some $0<r<1$. Then $y\colon\Sigma\to\C$ is an odd meromorphic function on $\Sigma$. On the domains 
$$\lambda^{-1}(\{\lambda\in\C\mid \lambda\bar\lambda<d\})\subset \Sigma\quad
\text{ and }\quad\lambda^{-1}(\{\lambda\in\C\mid \tfrac{1}{d}<\lambda\bar\lambda<d\})\subset \Sigma$$
of $\hat x$ and of $\hat a$, respectively, the function $y$ has no poles. Therefore, $\hat x$ and $\hat a$
can be written as
\[\hat x=y\underline x:= y\underline x\circ\lambda\;\;\; \text{ and }\;\;\;  \hat a=y\underline a:=y \underline a\circ\lambda,\]
for holomorphic functions
\[\underline x\colon \{\lambda\in\C\mid \lambda\bar\lambda<d\}\to \C\]
and 
\[\underline a\colon \{\lambda\in\C\mid \tfrac{1}{d}< \lambda\bar\lambda<d\}\to \C.\]

For $d>1$ fixed, let $\mathcal B_d$ be the Banach space of bounded holomorphic functions
on $$D_d:=\{\lambda\in\C\mid \lambda\bar\lambda<d\}$$
equipped with the supremum norm. Note that $\chi|_{\mathcal S}$ is a well-defined bounded map from the compact set $\mathcal S=\lambda^{-1}(S^1)$ into 
$\mathrm{Jac}(M)$. Let \[\chi_1=\chi(\xi) \quad \text{ and }\quad\chi_2 = \chi \left(\bar \xi+\tfrac{\tau_{spec}+1}{2}\right).\] Then
 there exist by Lemma \ref{NSsec} and Remark \ref{NSsecRem}   a 
$\delta>0$  and an open neighbourhood $\mathcal U_0\subset \mathcal B_d$ of the zero function  such that for 
all $\xi\in \mathcal S$, $\rho \in (-\delta, \delta)$,  and all
 $\hat x=y\underline x\circ\lambda$ with $\underline x\in \mathcal U_0 $  \[\chi(\xi)+\hat x(\xi) \in U_{\chi_1} \quad \text{ and }\quad \chi(\bar \xi+\tfrac{\tau_{spec}+1}{2})+\hat x(\bar \xi+\tfrac{\tau_{spec}+1}{2}))\in U_{\chi_2}.\]
In particular,
there is a real
analytic function
\begin{equation}\label{underlineA}
\alpha^\rho_{\underline x}\colon \mathcal S \to\C\end{equation}
such that
\[^\rho\nabla^{\chi(\xi)+\hat x(\xi),\alpha^\rho_{\underline x}(\xi)}\;\;\;\text{ and }\;\;\;\; \overline{^\rho\nabla^{\chi(\bar \xi+\tfrac{\tau_{spec}+1}{2})+\hat x(\bar \xi+\tfrac{\tau_{spec}+1}{2}),\alpha^\rho_{\underline x}(\bar \xi+\tfrac{\tau_{spec}+1}{2})}}\]
are gauge equivalent on $M$ for all $\xi\in \mathcal S .$\\

It also follows from  Lemma \ref{NSsec} and Remark \ref{NSsecRem} that for $\underline x\in\mathcal U_0\subset \mathcal B_d$ the map
$\alpha^\rho_{\underline x}$ has the same translational periods as $\alpha.$
Thus, for any $\underline x\in \mathcal U_0$  and any $\rho\in(-\delta, \delta)$ the function
\[\alpha^\rho_{\underline x}-\alpha\]
is real analytic and single-valued on both components $\mathcal S_1$ and $\mathcal S_2$ of $\mathcal S .$
Hence $\alpha^\rho_{\underline x}-\alpha$
is a holomorphic function defined on an open neighbourhood of $\mathcal S $ and odd by uniqueness in Lemma \ref{NSsec}.
Altogether, up to choosing $d>1$ smaller, $\alpha^\rho_{\underline x}$
 determines a holomorphic function
\[\underline a^\rho_{\underline x}\colon A_{d}:= \{\lambda\in\C\mid\tfrac{1}{d}<\lambda\bar\lambda<d\}\longrightarrow\C\]
satisfying  
\begin{equation}\label{underlineA2}
\alpha^\rho_{\underline x}-\alpha=y \underline a^\rho_{\underline x}\circ\lambda.
\end{equation}
for all $\xi \in \mathcal S$. 
In this vain,   let $\mathcal B_{\tfrac{1}{d},d}$  the Banach space of bounded holomorphic functions on the annulus
$A_d.$

\begin{lemma}\label{smoothB}

Let $\tau_{spec},\tau$ be as in Lemma \ref{no-spin}.
Then there exist $d>1,$ $\delta>0$ and an open neighbourhood $\mathcal U_0 \subset \mathcal B_d$ of the zero function
such that for every $\rho \in (-\delta, \delta)$ and every $\underline x\in\mathcal U_0$
the function
$\underline a^\rho_{\underline x}$ defined by \eqref{underlineA2} is bounded and holomorphic on $A_ d.$ Moreover, the map
\[(\rho,\underline x)\in(-\delta,\delta) \times\mathcal U_0\longmapsto \underline a^\rho_{\underline x}\in\mathcal B_{\tfrac{1}{d},d}\]
is smooth.
\end{lemma}
\begin{proof}
Consider the spectral curve $\Sigma=\C/\Lambda$ and its holomorphic maps $\lambda$ and $y$ to $\C P^1$ as in \eqref{ylambda}. 
By the assumptions of Theorem \ref{torus-spec-data}  there are no branch points  over $\lambda\in S^1 \subset\C P^1$.
Therefore,  the space of holomorphic functions $\mathcal B_{\tfrac{1}{d},d}$ from a small open annulus $A_d$ of $S^1$ can be identified 
with the space of odd holomorphic functions $\hat{\mathcal B}_{\tfrac{1}{d},d}$ from the corresponding neighbourhood $\lambda^{-1}(U)$ of $\mathcal S \subset \Sigma$ by multiplying with $y$.
Hence the map \[(\rho,\underline x)\in(-\delta,\delta) \times\mathcal U_0\longmapsto \underline a^\rho_{\underline x}\in\mathcal B_{\tfrac{1}{d},d}\]
is smooth if the map
\[(\rho,\underline x) \mapsto y \cdot a^\rho_{\underline x} \circ \lambda = \alpha^\rho_{\underline x}-\alpha  \]
to the Banach space $\hat{\mathcal B}_{\tfrac{1}{d},d}$ is smooth. To show this we construct a holomorphic function $F$ and a smooth map $h$ such that $$ y \cdot a^\rho_{\underline x} \circ \lambda = F \circ h(\rho, \underline x).$$

For $\tau_{spec}\in i\R^{>1}$ we decompose $\mathcal S=\mathcal S_1\cup \mathcal S_2$ into its two connected components.
Since we are working with odd maps, it suffices to study the behaviour of the involved functions on one connected component $\mathcal S_1$ or $\mathcal S_2$ only.
Denote by $\hat {\mathcal S}_i$ the preimage of $\mathcal S_i$ with respect to $\C\mapsto\C/\Lambda=\Sigma.$
Moreover, let
$$W:=\{0\}\times\left \{\left (\chi(\xi),\overline{\chi(\xi)},\chi(\bar\xi+\tfrac{\tau_{spec}+1}{2}),\overline{\chi(\bar \xi+\tfrac{\tau_{spec}+1}{2})} \right)\mid \xi\in \hat {\mathcal S}_1 \right \}\subset \C^5.$$

By analytic continuation there is a holomorphic function from an open neighbourhood $\hat W$ of $W$
$$F\colon \hat W\to\C$$
which is uniquely determined by 
\[F(\rho,\chi_1,\bar\chi_1,\chi_2,\bar\chi_2)=\tilde\alpha_1(\rho, \chi_1,\chi_2)-\tilde\alpha_1(0, \chi_1,\chi_2)\]
for all $(\rho,\chi_1,\bar\chi_1,\chi_2,\bar\chi_2)\in W,$ where $\tilde \alpha_1(\rho, \chi_1, \chi_2)$ are given by Lemma \ref{NSsec}.
Let 
$$\eta(\xi)=\bar\xi+\tfrac{1+\tau_{spec}}{2}$$
be the lift of the real structure constructed in Theorem \ref{torus-spec-data}
 and let $n\colon V\subset\C \to V$ be the holomorphic map determined by $n_{\mid \hat {\mathcal S}_1}=\eta_{\mid \hat {\mathcal S}_1}^{-1}$ for some open neighbourhood  $V$ of $\hat {\mathcal S}_1$  via analytic continuation. Note the map $n$ exists on $V$ because 
$\lambda|_{\mathcal S_k}: \mathcal S_k \rightarrow S^1$ is bijective for $k=1,2$. Then,

\begin{equation}
\begin{split}
h\colon\C\times {\mathcal B}_d&\to\C\times (\hat{\mathcal B}_{\tfrac{1}{d},d})^4\\(\rho,\underline x)&\mapsto
(\rho,\chi+\underline x\circ\lambda,\overline{(\chi+\underline x\circ\lambda)(\eta\circ n)},(\chi+\underline x\circ\lambda)\circ n,\overline{(\chi+\underline x\circ\lambda)\circ \eta })
\end{split}
\end{equation}
is smooth and we have by definition \eqref{underlineA2}
\begin{equation}\label{bigF}
\begin{split}
&y(\xi) \,\underline a^\rho_{\underline x}\circ\lambda(\xi)=F\circ h(\rho,\underline x),
\end{split}
\end{equation}
proving the statement.
\end{proof}
\begin{lemma}\label{lastlemma}
Using the notations of Lemma \ref{no-spin} and $r=\lambda([\tfrac{1}{2}])$. 
Then, there exist  $d>1$, $\delta>0$ and  an open set $$U=\lambda^{-1}\left(\{\lambda\in\C\mid 0<\lambda\bar\lambda<d; \lambda\neq r\} \right)$$ 
such that there are two
 smooth families of odd holomorphic maps
\[\rho\in(-\delta, \delta)\longmapsto \left( (\chi^\rho_+,\alpha^\rho_+)\colon U\subset \Sigma\longrightarrow \mathcal M_{dR}^1(M) \right)\]
and
\[\rho\in(-\delta, \delta)\longmapsto \left( (\chi^\rho_-,\alpha^\rho_-)\colon U\subset \Sigma\longrightarrow \mathcal M_{dR}^1(M) \right)\]
into the moduli space of flat line bundle connections $\mathcal M_{dR}^1(M)$ satisfying
\begin{enumerate}
\item for $\rho=0$ $(\chi^0_+,\alpha^0_+) = (\chi^0_-,\alpha^0_-)$ are the  spectral data in Theorem \ref{torus-spec-data};
\item for all $\rho\in(-\delta, \delta)$ the maps $\chi^\rho_\pm\colon U\to \mathrm{Jac}(M)$ extend holomorphically through $[0]=\lambda^{-1}(0) $ and $[\tfrac{1}{2}]=\lambda^{-1}(r)\in\Sigma$;
\item $\alpha^\rho_\pm$ has first order poles at $[0],[\tfrac{1}{2}]\in\Sigma$, and the expansion of the pole at $\lambda = [\tfrac{1}{2}]$ satisfies  \eqref{a_spin_expansion} with the respective sign; 
\item for all $\rho\in(-\delta,\delta)$ and for all $\xi\in \mathcal S$ the connections
\[^\rho{\hat\nabla^{\chi^\rho_\pm(\xi),\alpha^\rho_\pm(\xi)}}\;\;\;\text{ and }\;\;\;\overline{^\rho{\hat\nabla^{\chi^\rho_\pm\left(\bar\xi+\tfrac{1+\tau_{spec}}{2}\right), \; \alpha^\rho_\pm\left(\bar\xi+\tfrac{1+\tau_{spec}}{2}\right)}}}\]
are gauge equivalent on the punctured torus  $M\setminus\{P_1,..,P_4\}$.
\end{enumerate}
\end{lemma}
\begin{proof}
Using Lemma \ref{NSsec} and Lemma \ref{smoothB}
the proof is analogous to the proof of Theorem 4.2 in \cite{HeHeSch}. It requires first to apply the implicit function theorem to guarantee $(4)$ and $\alpha^\rho_\pm$ having a first order pole in $\lambda= [0].$ The remaining parameters are used to apply the implicit function theorem again to obtain the asymptotic  \eqref{a_spin_expansion}  at $\lambda = [\tfrac{1}{2}].$ Since the residue  in \eqref{a_spin_expansion}  vanishes for $\rho = 0,$ we have a choice of sign here.
The main difference is that in contrast to Theorem 4.2 of \cite{HeHeSch}
no extrinsic closing condition is required.  
This allows us to fix the branch point $r$ of the spectral curve for all $\rho$. 
 \end{proof}
 \begin{remark}\label{flow-identity}
 The two connections $^{\rho}\hat\nabla^{\chi,\alpha}$  and $^{-\rho}\hat\nabla^{\chi,\alpha}$ \eqref{connection1form} are gauge equivalent by a diagonal gauge on the punctured torus $M\setminus\{P_1,..,P_4\}$.
The uniqueness part of the implicit function theorem in the proof of Lemma \ref{lastlemma} and the  sign change in \eqref{a_spin_expansion} then give
 \[(\chi^\rho_+,\alpha^\rho_+)=(\chi^{-\rho}_-,\alpha^{-\rho}_-)\]
 for all $\rho\in(-\delta,\delta).$
 \end{remark}
 
 \section{Existence of  negative real sections}\label{sec:erns}
 In this section we prove the existence of negative, non-admissible, real holomorphic sections in 
 Deligne-Hitchin moduli spaces of Riemann surfaces of high genus. We show that each set of the initial data (Lemma \ref{lastlemma}) give rise to
two $\C^*$-families of flat connections on the 4-punctured 
 sphere (Proposition \ref{hatXfamily}). These families induce real holomorphic sections of the Deligne-Hitchin moduli spaces of certain coverings of the 4-punctured sphere. 
  
\begin{definition}
Let $\rho\in\R.$ The quadruple $(\Sigma, \lambda, \chi, \alpha)$ is called $\rho$-spectral data, if
\begin{itemize}
\item $\Sigma$ is the spectral curve and $\lambda: \Sigma \rightarrow \C P^1$ as in Lemma \ref{lastlemma};
\item $\chi : \tilde \Sigma := \lambda^{-1}(D_d) \subset \Sigma \rightarrow \text{Jac}(M)$ is an  odd holomorphic map for some $d>1$;
\item $(\chi,\alpha) : \tilde \Sigma \subset \Sigma \rightarrow \mathcal M^1_{dR}(M)$ is an odd meromorphic map such that $(\chi,\alpha)$ satisfies the condition \eqref{a_spin_expansion} for $\rho.$
\end{itemize}
 \end{definition}

 Let $M=\C/\Gamma$ be a rectangular torus given by the algebraic equation
$$y^{2}=\frac{(z-z_0)(z-z_1)}{(z-z_2)(z-z_3)}$$
for pairwise distinct points $z_0,..,z_3\in\C$. 
Without loss of generality we assume that the branch points of $z$ are the half-lattice points 
$[0],[1],[1+\tau]$ and $[\tau]$. For $q\in\mathbb N^{>1}$
consider the compact  Riemann surface $M_q$ of genus $q-1$ defined by the algebraic equation
\begin{equation}\label{algEQ}y^{q}=\frac{(z-z_0)(z-z_1)}{(z-z_2)(z-z_3)}.\end{equation}
Naturally, the surface $M_q$  is equipped with  a $q$-fold covering $z$ to the Riemann sphere. If $q$ is even, it further admits a $\tfrac{q}{2}$-fold covering to the torus $M$ given by
\begin{equation}\label{hatpi}\pi_q\colon  M_q\to M; (y,z)\mapsto (y^{\tfrac{q}{2}},z)\end{equation}
 totally branched at $[0],[1],[1+\tau],[\tau]\in M$.

\begin{proposition}\label{hatXfamily}
Let $p,q\in\mathbb N$ be coprime integers with $q$ even and $\rho=-p/q\in(-\tfrac{1}{2},0)$. Let $(\Sigma, \lambda, \chi,\alpha)$ be $\rho$-spectral data provided by Lemma \ref{lastlemma} %
and $\pi_q\colon M_q \to M$ be as in \eqref{hatpi}. Then the $\rho$-spectral data induce a holomorphic family
of smooth flat $\SL(2,\C)$-connections \[\nabla^\lambda,  \quad  \lambda\in D^*_{d} = \{\lambda\mid0<\lambda\bar\lambda<d\}\]
on $M_q$ with asymptotic expansion 
\begin{equation}\label{lambda-extension}
\nabla^\lambda=\lambda^{-1}\Phi+\nabla+\lambda \Phi_1+ \text{ higher order terms in } \lambda\end{equation}
around $\lambda = 0$ such that  $(\dbar^\nabla,\Phi)$ is a stable Higgs pair. Furthermore, under the assumptions of Lemma \ref{no-unit-spin} the connections $\nabla^\lambda$ are irreducible for all $\lambda\in D^*_{d}$ for $\rho\sim0$ small. 
\end{proposition}
\begin{proof}
The $\rho$-spectral data  define a family of connections 
$^\rho\tilde\nabla^{\chi(\xi),\alpha(\xi)}$ (see \eqref{connection1form}) on the $4$-punctured sphere $\C P^1 \setminus \{p_1,...p_4\}$ parametrised by $\xi \in \tilde\Sigma$. 
Because $\chi$ and $\alpha$ are odd maps, we can define
$$\tilde \nabla^\lambda:=\nabla^{\lambda(\xi)}=   \;^\rho\hat\nabla^{\chi(\xi),\alpha(\xi)}$$
via the hyper elliptic involution $\lambda$ on $\Sigma$. 
The condition that $\alpha$ satisfies \eqref{a_spin_expansion} ensures that there exist a global $\lambda$-dependent gauge $g$ such that $\tilde \nabla^\lambda. g$ extends holomorphically through the points $\xi$, where $\chi(\xi)$ is a lattice point.  
The pull-back of $\tilde \nabla^\lambda .g$ to $M_q$ then has apparent singularities in $\pi_q^{-1}(\{p_1, ...,p_4\})$ and is therefore gauge equivalent to a holomorphic family of smooth connections $\nabla^\lambda$ over $M_q$.
As in \cite[Section 3.3]{HeHeSch} the Higgs pair at $\lambda=0$ is stable, see also \cite[Section 2.4]{HeHe}. \\

It remains to show  irreducibility of $\nabla^\lambda$ for all $\lambda\in\C^*$. 
For $\rho \neq 0$ a connection $\tilde \nabla^{\lambda_0}$ on the 4-punctured sphere  is reducible if and only if the corresponding $d-\chi_0d\bar w+\alpha_0 dw$ is a non-trivial spin connection by  \cite[Section 2.4]{HeHe}. 
 Because $M_q\to \C P^1$ is a cyclic covering, the same holds for  $\nabla^\lambda$.
At $\rho = 0$ and under the assumptions of Lemma \ref{no-unit-spin} $d-\chi_0d\bar w+\alpha_0 dw$ is never a non-trivial spin connection.
By continuity of the flow, this remains true for $\rho \sim 0$ giving irreducibility for $\nabla^\lambda$ for $\rho \neq 0$.
\end{proof}

\begin{remark}\label{cp1family}
In the following we use the family $\tilde\nabla^\lambda$ of flat connections on the 4-punctured sphere because these connections are
irreducible for all $\rho$. Therefore, the notion of positive and negative real sections is well-defined and preserved by continuous deformations. The Proposition also holds for $\rho>0$. Due to the stability properties (Remark \ref{Sign}) this implies that
the bundle type at $\lambda=0$ is $\mathcal O(1)\oplus\mathcal O(-1)\to\C P^1$ instead of the trivial holomorphic bundle, see \cite[Section 2.4]{HeHe}.
The conformal conformal Gau\ss \; changes the underlying holomorphic bundle at $\lambda=0$ also to $\mathcal O(1)\oplus\mathcal O(-1)\to\C P^1$ .
\end{remark}

\begin{theorem}\label{MainTheorem}
There is a $g_0\in \mathbb N$ such for every $g> g_0$ there exists a Riemann surface of genus $g$
admitting a negative real holomorphic section in its Deligne-Hitchin moduli space which is not a twistor line.
\end{theorem}
\begin{proof}
For $\rho< 0$ consider $\rho$-spectral data satisfying Lemma \ref{no-unit-spin} and
let $\delta>0$ as in Lemma \ref{lastlemma}. Let $g_0\in\mathbb N$ such that $\tfrac{1}{g_0+1}<\delta.$
Then, for  $q\geq (g_0+1)$ and $\rho=-\tfrac{1}{q}$, we obtain two holomorphic families of flat connections on the $q$-fold covering $M_{q}\to \C P^1$ (of genus $g=q-1$)
by Lemma  \ref{lastlemma} and  Proposition \ref{hatXfamily}. Both families give rise to  local (in $\lambda$) sections of the Deligne-Hitchin moduli space
on $\{\lambda\in\C\mid \lambda\bar\lambda<1+\epsilon\}\subset\C P^1$ by the properties (2) and (3) in Lemma  \ref{lastlemma}.
By property (4) the sections are real with respect to the involution $\mathcal T$ and hence extend to global sections
\[s_\pm\colon \C P^1\longrightarrow \mathcal M_{DH}(M_q).\]
These sections are stable by Proposition \ref{hatXfamily}, compare with Remark \ref{cp1family}.


We claim either $s_+(\pm\rho)$ or $s_-(\mp\rho)$ is negative.
Assume that $s_+(\rho)$ is positive, we want show that in this case $s_-(-\rho)$ must be negative. Consider the local lift of the section $s_+(\rho)$ on $\{\lambda\in\C\mid 0<\lambda\bar\lambda<1+\epsilon\}$ to the space of flat connections
on the 4-punctured sphere given by
$$ \;^\rho\tilde \nabla^{\lambda(\xi)} = \;^\rho{\tilde\nabla^{\chi(\xi),\alpha(\xi)}}.$$
For $\rho<0$ the corresponding $\lambda$-connections extend to $\lambda=0$ as a stable nilpotent parabolic Higgs field on the trivial holomorphic bundle, see
 \cite[Section 2.4]{HeHe}. For $\rho\sim 0$  the connections on the 4-punctured sphere are irreducible by Lemma \ref{no-unit-spin}. 
 
Applying the conformal Gauss map construction (Section \ref{subsubdual}) there exist a negative lift $^\rho\underline{\nabla^\lambda}$ of $s_+(\rho)$ which is irreducible for $\lambda \neq 0$
but has an unstable parabolic Higgs bundle at $\lambda=0$ with respect to the parabolic weight $\rho<0$. Moreover, the underlying holomorphic bundle  is $\mathcal O(1)\oplus\mathcal O(-1)\to\C P^1.$

On the 4-punctured sphere, the connections $^{\rho}\underline{\nabla^\lambda}$ are irreducible for all $\lambda\in\C^*$ (respectively for all $\lambda\in\C^*\setminus\{\lambda(\xi_0)\}$ if $\rho=0$). Therefore the sign in \eqref{realeqsecsign} does not change within the  continuous $\rho$-deformation.  Hence, at $-\rho$ the family $^{-\rho}\underline{\nabla^\lambda}$ is negative. Since $-\rho>0$, and because the underlying holomorphic bundle is $\mathcal O(1)\oplus\mathcal O(-1)\to\C P^1,$ the parabolic Higgs pair of $^{-\rho}\underline{\nabla^\lambda}$ at $\lambda = 0$ is stable, see again  \cite[Section 2.4]{HeHe}.
Moreover, the parabolic structure is unstable at the branch point $\lambda(\xi_0)$  for $-\rho>0$, see Remark \ref{Sign}.
By Remark \ref{flow-identity}
\[(\chi^\rho_+,\alpha^\rho_+)=(\chi^{-\rho}_-,\alpha^{-\rho}_-),\]
showing that $^{-\rho}\underline{\nabla^\lambda}$ is a lift of $s_-(-\rho).$ Therefore, $s_-(- \rho)$ is a negative section on the $4$-punctured sphere. Because negativity and stability is preserved by 
 pull-backs $s_-(-\rho)$ gives also a negative section on $M_q$.



It remains to prove that these negative real sections are not twistor lines. Solutions to the self-duality equations on punctured Riemann surfaces have been studied in \cite{Si1}. For the $4$-punctured sphere  with all parabolic weights being $\tilde \rho = \tfrac{1}{4}$ (and parabolic Higgs fields with nilpotent residues as in our case) the solutions correspond to smooth and reducible solutions on the torus $M \xrightarrow[]{2:1} \C P^1 \setminus\{p_1, ..., p_4\}$. The smooth dependence of the solutions on the parabolic weight $\tilde \rho =  \tilde\rho_i$ (see \cite{KiWi}) gives that the solutions of the self-duality equations for $\tilde \rho \sim \tfrac{1}{4}$ are close to these reducible solutions (with respect to the supremum norm after applying appropriate gauge transformations).
The negative real sections $s$ constructed have initial data at $\tilde \rho = \tfrac{1}{4}$ which do not solve the self-duality equations globally on $M$ and can therefore not be twistor lines for $\tilde\rho\sim \tfrac{1}{4}$.\end{proof}

\section{Higher solutions of the self-duality equations}\label{sec5}
In Theorem  \ref{MainTheorem} we constructed negative real sections fo the Deligne-Hitchin moduli space which are not twistor lines. In this section we want to give a geometric interpretation of these sections in terms of Willmore surfaces in $3$-space. To do so, we show that the sections given by Theorem  \ref{MainTheorem}  give solutions to Hitchin's self-duality equations on an open and dense subset of the Riemann surface $M$.

It is well-known that self-duality solutions for $G_\C=\SL(2,\C)$ correspond to equivariant harmonic maps into hyperbolic 3-space, see
\cite{Do, Si0}. Given a solution $(\nabla,\Phi,h)$ on a Riemann surface $M$ for a unitary connection $\nabla$ with respect to the hermitian metric $h$ and the Higgs field $\Phi$ we consider the flat connection
\begin{equation}\label{defnab1}\nabla^1=\nabla+\Phi+\Phi^*.\end{equation} 
 Fix a point $p\in M$  and a unitary frame $\Psi_p$ at $p$ of the rank 2 vector space $V_p=\C^2$.
 Through parallel transport we obtain an equivariant frame 
\[\Psi:=(e_1,e_2)\] 
on the universal covering $\widetilde M\to M$. Then 
\[H\colon q\in \widetilde M\longmapsto \begin{pmatrix} h(e_1(q),e_1(q)) & h(e_1(q),e_2(q))\\ h(e_2(q),e_1(q)) & h(e_2(q),e_2(q))\end{pmatrix}\in\{H\in\SL(2,\C)\mid \bar H^T=H\}\]
 is the corresponding equivariant harmonic map   into the hyperbolic 3-space
\[ \mathbb H^3=\SL(2,\C)/SU(2)=\{H\in\SL(2,\C)\mid \bar H^T=H\}^+,\]
where  $\{H\in\SL(2,\C)\mid \bar H^T=H\}^+$ is the component of $\{H\in\SL(2,\C)\mid \bar H^T=H\}$ containing the identity matrix.
Without loss of generality we assume in the following that $h$ is the standard metric, i.e., $A^*=\bar A^T$ and
 the map $H$ is given by
\[H=\bar\Psi^{T}\Psi.\]
\begin{remark}
In the following, we will consider $\{H\in\SL(2,\C)\mid \bar H^T=H\}$ as the union of two copies of the  hyperbolic 3-space, and $H$ maps
 into this union. Depending 
on the component, a harmonic map $H$ will therefore be given by
\[H=\pm\bar\Psi^{T}\Psi.\]
\end{remark}

\begin{definition}
Let $ M$ be a compact Riemann surface, and $U\subset M$ be an open dense subset. Let
$(\nabla,\Phi,h)$ be a solution of the self-duality equations on $U$. We say  that the solution converges to $\infty$ for  $p\to\partial U$ if the corresponding equivariant map $H$ converges to the boundary of $\mathbb H^3$.\end{definition}
\begin{remark}
This definition is well-defined, since going to the boundary of $\mathbb H^3$ is invariant under
the action of $\SL(2,\C).$ With respect to the matrix model the condition just means that the operator norm
(with respect to the standard hermitian metric on $\C^2$)
of the matrix $H$ goes to $\infty$ as $p\to\partial U$.
\end{remark}
There is a further useful description of $\mathbb H^3$  (see also \cite[$\S$ 2]{BaBo}):
consider a totally geodesic 2-sphere in the 3-sphere $S^2\subset S^3.$ The complement $S^3\setminus S^2$
consists of two $3$-dimensional hemispheres, and each of them can be equipped with the hyperbolic metric. For explicit computations
we use the stereographic projection of $S^3$ determined by a point $p\in S^2\subset S^3$ such that $S^2\setminus\{p\}$ is mapped to
$\R^2\times\{0\}.$ The hyperbolic  metric on the two (3-dimensional) half planes is then given by 
\[g_{(x,y,z)}= \frac{1}{z^2} (dx\otimes dx+dy\otimes dy+dz\otimes dz).\] 
The isometry between the matrix and the half-plane model is given by
\begin{equation}\begin{split}
\{H\in\SL(2,\C)\mid \bar H^T=H\}\longrightarrow \R^3\setminus (\R^2\times\{0\});\;\;\; \\ \begin{pmatrix} x_0+x_3 & x_1-ix_2 \\ x_1+ix_2 & x_0-x_3 \end{pmatrix}\longmapsto \frac{1}{x_0-x_3} \begin{pmatrix} x_1\\x_2\\1\end{pmatrix}.\end{split}\end{equation}
In this setup, a map goes through the $\infty$-boundary of $\mathbb H^3$ (away from the base point $p$ of the stereographic projection) in first order, if the third component $z$ in the half plane model has 0 as a regular value. 
The action of the isometry group $P\SL(2,\C)$ of the hyperbolic 3-space extends to an action on $S^3$ by conformal transformations. Restricted to the boundary $S^2=\C P^1$ we obtain the action of the Moebius group. In the light cone model, the conformal transformations of $S^3$ are given by the (standard) inclusion
\[
\SL(2,\C)=\SO(3,1)\longrightarrow \SO(4,1)
\]
given by the action on $\{H\in\SL(2,\C)\mid \bar H^T=H\}$
\begin{equation}\label{SLonS3}(h,g)\mapsto \bar g^THg.\end{equation}

We then have the following geometric interpretation of Theorem \ref{MainTheorem}.

\begin{theorem}\label{higherharmonic}
Up to choosing $\rho\sim0$ smaller,
every real section $s$ of the Deligne-Hitchin moduli space $\mathcal M_{DH}(M_q)$ constructed in Theorem \ref{MainTheorem}  gives rise to a solution of
the self-duality equations on an open dense subset $U \subset M_q$ which converge to $\infty$ as $p\to\partial U$.
The boundary $\partial U$ is given by closed regular curves.

The corresponding equivariant harmonic map $f \colon U\to \mathbb H^3$ extends through the boundary $\partial U$ and gives a smooth map $ \tilde f\colon \widetilde M_q\to S^3$  on the universal covering $\pi\colon \widetilde M_q\to M_q$. In particular,  $\tilde f$ is  a branched Willmore surface equivariant with respect to the  monodromy representation of the connection
$\nabla^1$ acting on $S^3$ via \eqref{SLonS3} by conformal transformations.  
\end{theorem}

The proof of the Theorem \ref{higherharmonic} relies on properties of loop group factorisations. We thus briefly recall some basic definitions and concepts for loop groups first, for details see \cite[$\S3$ and $\S8$]{PS}.
\begin{definition}
A loop is a smooth map $\gamma\colon S^1\to \SL(2,\C)$.  The loop group $\Lambda\SL(2,\C)$ is the set of all loops $\gamma.$ 
 \end{definition}
$\Lambda\SL(2,\C)$ has a natural Lie group structure given by point-wise multiplication together with a compatible differentiable structure.
  A loop is called positive, if it extends to a holomorphic map from the unit disc $D_1\subset\C$ to $\SL(2,\C)$, and negative, if it extends holomorphically to $\C P^1\setminus\bar D_1.$ The space of positive (respectively negativ) loops is denoted by $\Lambda^+\SL(2,\C)$
  (respectively $\Lambda^-\SL(2,\C)$).
  \begin{definition}
  The {\em big cell} of $\Lambda\SL(2,\C)$ is the open and dense subset of $\Lambda\SL(2,\C)$  whose elements are given by the product of a positive and a negative loop, i.e., \[\gamma = \gamma_+ \gamma_-.\] 
  \end{definition}

 The factorisation of a given element $\gamma$ in the big cell into its  positive and negative loop is called Birkhoff factorisation. It extends to the whole loop group by allowing a middle term $\mathcal J=\text{diag}(\lambda^k,\lambda^{-k})$, i.e., $ \gamma = \gamma_+\mathcal J \gamma_-$ for every $\gamma \in \Lambda\SL(2,\C).$  For  $\gamma$ lying in the big cell the factorisation  is uniquely determined by fixing $\gamma_-(\infty)=\text{Id}.$  Equivalently, the factorisation $\gamma = \gamma_+ \gamma_-$ is unique up to multiplying a constant $B\in\SL(2,\C)$, i.e., up to $\tilde \gamma_+=\gamma_+B$ and $\tilde \gamma_-=B^{-1}\gamma_-.$
  
A (real-analytic) loop $\gamma$ can be interpreted as the transition function of a holomorphic rank $2$ vector bundle $V_{\gamma}$ trivialised over $U_0 := D_{1+\epsilon}$ and $U_\infty := \C P^1\setminus\overline D_{\tfrac{1}{1+\epsilon}}.$ Because $\gamma$ maps into SL$(2, \C)$ the bundle $V_\gamma$ has trivial determinant. Moreover, if $\gamma = \gamma_+ \gamma_-$ lies in the big cell of $\Lambda\SL(2,\C),$ then the two  holomorphic frames given by $\gamma_+$ on $U_0$  and $\gamma_-^{-1}$ are trivialising frames of $V_{\gamma},$ i.e,  $V_{\gamma}$ is the trivial holomorphic $\C^2$-bundle over $\C P^1.$   In the context of this paper a real-analytic loop 
 is obtained from the gauge transformations $g(\lambda)$ evaluated at $x \in M$, which we denote by $g_x(\lambda)$. In this case we have $U_0 = \C$ and $U_\infty = \C P^1\setminus \{0\}.$ Because the initial condition of the generalised Whitham flow is particularly well understood, we can show that the possible bundle types $V_{g_x}$ are very limited for $\rho$ small. 

\begin{proposition}\label{jump-up-1}
Let $x\in M_q$, $\rho\sim 0$ and $\nabla^\lambda$ be the lift of a real section  $s$ of the Deligne-Hitchin moduli space over the compact Riemann surface $M_q$ constructed in Theorem \ref{MainTheorem} with gauge transformations $g(\lambda)$ satisfying  \eqref{realeqsecsign} with the minus sign. Then  the bundle $V_{g_x}$ 
is either trivial or $V_{g_x} \cong \mathcal O(1)\oplus\mathcal O(-1)$. The set 
$$W := \{x \in M_q | V_{g_x} \cong \mathcal O(1)\oplus\mathcal O(-1) \}$$
 is locally given by the zero locus of a differentiable real-valued function.
\end{proposition} 
\begin{proof}

For $\rho=0$ the initial data are given by spectral genus 1 solutions of the cosh-Gordon equation. The Birkhoff factorisation in this case is well understood and has the claimed properties, see \cite{BaBo}. Moreover, the bundles $V_{g_x}$ are trivial  at the branch points of $M_q\to\C P^1$. 

Since $\Lambda \SL(2, \C )/\Lambda^+ \SL(2, \C )$ is stratified by the induced bundle type, see \cite[Theorem 8.6.3]{PS}, the property that  $V_{g_x}$ is either trivial or $\mathcal O(1)\oplus\mathcal O(-1)$  is an open condition and therefore preserved for $\rho\sim 0.$

By Lemma \ref{realrfun}  the set $W$ is the zero set of a real valued function $r.$ 
\end{proof}
\begin{remark}
We will see below that $W$ is in fact a (compact) submanifold, i.e., the union of smooth closed curves.
\end{remark}

To proceed we need the following lemma.
It is well-known for holomorphic maps, but
we are not aware of any reference which deals with the case of $\mathcal C^k$-maps.
\begin{lemma}\label{normal-form-O1}
For $k\in\mathbb N^{>0}$ and a manifold $M$ let
\[y\in M\longmapsto \left (g_y\colon \C^*\longrightarrow \SL(2,\C) \text{ holomorphic } \right)\]
be a $\mathcal C^k$-map into the loop group such that the corresponding rank 2 bundles $V_{g_y}$ over $\C P^1$ are either trivial or $\mathcal O(1)\oplus\mathcal O(-1).$ Let $x\in W= \{y \in M | V_{g_y} \cong \mathcal O(1)\oplus\mathcal O(-1) \}.$ Then there exists
an open neighbourhood $U\subset M$ of $x$, a $\mathcal C^k$-function $r\colon U\to\C$ and $\mathcal C^k$-maps
\[y\in U\longmapsto(  h_y^+\colon\C\longrightarrow SL(2,\C))\]
into the positive loop group and 
\[y\in U\longmapsto(  h_y^-\colon\C P^1\setminus\{0\}\longrightarrow SL(2,\C))\]
into the negative loop group such that
for all $y\in U$
\begin{equation}\label{ghdr}
g_y=h_y^+\begin{pmatrix} \lambda^{-1}& r(y)\\0 & \lambda\end{pmatrix} h_y^-.\end{equation}

In particular, $r(y)\neq 0$ is equivalent to $V_{g_y}$ being trivial. 
\end{lemma}
\begin{proof}
The last assertion follows from the fact that the Birkhoff splitting of  the middle term in \eqref{ghdr} for $r \neq 0$ is given by
\[\begin{pmatrix} \lambda^{-1}& r\\0 & \lambda\end{pmatrix} =\begin{pmatrix} r& 0\\  \lambda&\tfrac{1}{r}\end{pmatrix}\begin{pmatrix} \tfrac{1}{r}\lambda^{-1}& 1\\  -1&0\end{pmatrix}\]
 implying $V_{g}$ being trivial.

For $y\in M$ consider the $\mathcal C^k$-family of holomorphic bundles $\tilde V_y$ given by the cocycle
\[\lambda\mapsto\lambda g_y(\lambda).\]
In particular, $\tilde V_x$ is of type $\mathcal O \oplus \mathcal O(2)$ for $x \in W$ and its determinant bundle is $\mathcal O(2)\to\C P^1.$ 

By Riemann Roch the space of holomorphic sections of $\tilde V_y$  is complex 4-dimensional for every $y\in M$.  Therefore, the space of holomorphic sections of $\tilde V_y$ defines a
 a rank four  $\mathcal C^k$-vector bundle $\hat V$ over $M$. 
 
 Let $x \in W.$ Choose two local sections $s_1,s_2$ of $\hat V$ in a neighbourhood $U$ of $x$ such that the two holomorphic sections $s_1(x),s_2(x)$ of $\tilde V_x$
have determinant 
\[\det\left(s_1(x), s_2(x)\right) = \lambda(\lambda-1)\in H^0(\C P^1,\mathcal O(2)).\]

For $y \in U$ fixed, $s_i(y)$ are holomorphic sections of $\tilde V_y \rightarrow \C P^1.$  With respect to the local trivialisation of the bundle $\tilde V_y$ over $U_+ = \C$ and  $U_- = \C P^1 \setminus \{0\}$ each $s_i(y)$ is determined by a pair of $\mathbb C^2$-valued holomorphic functions $f^i_\pm \colon U_\pm \rightarrow \C^2$. By definition, the frames
$$F_+(y)=\left (f^1_+(y),f^2_+(y)\right) \quad \text{and} \quad F_-(y)=\left (f^1_-(y),f^2_-(y)\right)$$ 
of $\tilde V_y$ over $U_+$ and $U_-$, respectively, satisfy
\[F_+(y)=\lambda g_y F_-(y)\]
for all $y \in U$. We will omit the argument $y$ of $F_\pm$ in the following. The determinants of $F_\pm$ satisfy
\[\det(F_+)=\lambda^2\det(F_-)\]
with $\det(F_+)$ being a polynomial of degree 2 in $\lambda.$ By continuity and by choosing $U$ small enough $\det(F_+)$ has two simple zeros (close to $\lambda = 0$ and $\lambda =1$) for all $y\in U$.\\

Assume that there exist $\mathcal C^k$-families of $\mathfrak{gl}(2,\C)$-valued polynomials 
$$P = P_0 + \lambda P_1 + \lambda^2 P_2 \quad \text{and} \quad Q= Q_0 + \lambda^{-1}Q_1 + \lambda^{-2}Q_2$$ 
with the property that $P$ and $Q$ are invertible for generic $\lambda$ such that 
\begin{equation}\label{PolyPQ}\tilde F_+=F_+ P^{-1}\;\;\;\;\text{ and }\;\;\;\;\; \tilde F_-=F_- Q^{-1}\end{equation}
are holomorphic maps  on $\mathbb C^*$ into $\SL(2,\C)$. In particular we have
\[\det (P_x)= \det (F_+(x)) = \lambda (\lambda -1) \quad \text{and} \quad \det (Q_x) = 1- \lambda^{-1}.\]
Then
\begin{equation}\label{PQm1}\tfrac{1}{\lambda}PQ^{-1}=\tilde F_+^{-1}g\tilde F_-\end{equation}
is a $\mathcal C^k$-family of holomorphic maps $\C^*\to\SL(2,\C).$

Since $\tilde F_+^{-1}g\tilde F_-$ is $\SL(2,\C)$-valued and holomorphic in $\lambda \in \C^*,$
the coefficients of the matrix $\tfrac{1}{\lambda} P Q^{-1} \in \mathfrak{gl}(2,\C)$ extend holomorphically through the zeros of $\det Q.$ 
Writing these coefficients in terms of the coefficients of $P$ and $Q$ and using 

\begin{equation}\label{detPQ}
1 = \det \tilde F_+^{-1}g\tilde F_-= \det \left (\tfrac{1}{\lambda}PQ^{-1}\right)
\end{equation}
 we can compute

\begin{equation}\label{PQ}
\tfrac{1}{\lambda}PQ^{-1}=\lambda^{-1}r_{-1}+r_0+\lambda r_1.
\end{equation}
for some matrices $r_i\in\mathfrak{gl}(2,\C).$  

This gives $\det r_{-1} = \det r_1 =0$ for all $y \in U.$
If $r_1 = 0$, then $\tfrac{1}{\lambda} PQ \tilde F_-$ is a negative loop and $g = \tilde F (\tfrac{1}{\lambda} PQ \tilde F_-)$ is the Birkhoff factorisation of $g$ into positive and negative parts. Therefore, $r_1(x)\neq 0$ and $r_{-1} \neq 0$ for $x \in W$. A similar argument gives  $\det (r_0(y))=0$ for all
$y\in U$.  Otherwise, $\tfrac{1}{\lambda}PQ^{-1} r_0^{-1}$ lies in the big cell.

To obtain \eqref{ghdr} we write $r_i$ with respect to a suitable basis of $\tilde V_y.$ Firstly, we can find suitable $\mathcal C^k$-maps
\[h_{1}, \tilde h_2\colon U\longrightarrow \SL(2,\C)\]
 such that
\[h_1 r_{-1} \tilde h_2=\begin{pmatrix} 1 &0\\0&0\end{pmatrix}.\]

Equation \eqref{detPQ} then implies
\[h_1 r_1 \tilde h_2=\begin{pmatrix} a &b\\c&1\end{pmatrix}.\]

Choosing

\[h_2 = \begin{pmatrix} 1 &-c\\0&1\end{pmatrix} \tilde h_2\]
gives 

\[h_1 r_{-1} h_2=\begin{pmatrix} 1 &0\\0&0\end{pmatrix} \quad \text{and} \quad h_1 r_1 h_2=\begin{pmatrix} 0 &\hat r\\0&1\end{pmatrix}\]

for some $\mathcal C^k-$function $\hat r: U \rightarrow \C.$ Moreover, due to \eqref{detPQ},

\[ h_1 r_0 h_2=\begin{pmatrix} 0 &r\\0&0\end{pmatrix}\]
for some $\mathcal C^k$-function $ r\colon U \rightarrow \C.$  The desired factorisation of $g$ is then given by

\[g= \tilde F_+ h_1^{-1}\begin{pmatrix} 1&\hat r(y)\\0&1\end{pmatrix}\begin{pmatrix}
\lambda^{-1} & r(y)\\0&\lambda
\end{pmatrix}
h_2^{-1}\tilde F_-^{-1},\]
i.e.,
$$h^+_y :=   \tilde F_+ h_1^{-1}\begin{pmatrix} 1&\hat r(y)\\0&1\end{pmatrix}\quad \text{ and } \quad h^-_y :=h_2^{-1} \tilde F^{-1}_-.$$

To complete the proof it remains to show the existence of the polynomials $P$ and $Q$ in \eqref{PolyPQ}. This is done pointwise in $y\in U$ in the following Lemma \ref{detcancelation}. The  $\mathcal C^k$ dependency in $y$ follows from the fact that the determinants of $F_+$ and $F_-$ have only simple zeros on $U$.
\end{proof}

\begin{lemma}\label{detcancelation}
Let $F\colon\C\to\mathfrak{gl}(2,\C)$ be a holomorphic map such that $\det(F)\neq 0$ is a polynomial of order $\leq d.$ Then there is a polynomial
\[P\colon\C\longrightarrow \mathfrak{gl}(2,\C)\]
of degree $\leq d$
such that $\det(P)=\det(F)$ and
$FP^{-1}$ extends holomorphically through the zeros of $\det(F).$
\end{lemma}
\begin{proof}
We prove the Lemma by induction over the degree $d$ of the polynomial $\det(F).$ 
For $d=0$ we can choose $P = $Id and the assertion holds trivially. 

Assume that for every $F$ such that $\det (F)$ has degree $k \leq d$  there exists a polynomial $P_{k}$ of degree $k$ with the desired properties. For $F$ such that $\det(F)$ is of degree $d+1$ let 
$\lambda_0$ be a zero of $\det(F).$ There are two cases to consider. 

The first is $F(\lambda_0) = 0.$ In this case choose $\tilde F: \C \rightarrow\mathfrak{gl}(2, \C)$ such that 
$$ F =  \tilde F \cdot (\lambda- \lambda_0) \text{Id}.$$
Then 
$$\det(F) = (\lambda- \lambda_0)^2 \det(\tilde F),$$
and $\det(\tilde F)$ is of degree $d-1.$ By assumption we can find a polynomial $P_{d-1}$ of degree $d-1$ with $\det(\tilde F) =  \det(P_{d-1})$ such that $\tilde F P_{d-1}^{-1}$ extends holomorphically through the zeros of $\det (\tilde F).$
Then 
$$P=(\lambda-\lambda_0) \text{Id} \cdot P_{d-1}$$
satisfies that $F P^{-1}$ extends holomorphically through the zeros of $\det (F).$

In the second case $F(\lambda_0)\neq0$, and $F(\lambda_0)$ has a one-dimensional kernel $L$. Decompose $\C^2=L\oplus\tilde L$ for some complementary line $\tilde L$ and choose $\tilde F$ such that 
$$ F =  \tilde F \cdot \begin{pmatrix} \lambda-\lambda_0 &0 \\ 0&1   \end{pmatrix}$$
with respect to the splitting.  Then $\det(\tilde F)$ is of degree $d$ and there exists a suitable polynomial $P_d$ by assumption. By choosing
$$P(\lambda)=P_d \cdot \begin{pmatrix} \lambda-\lambda_0   &0 \\ 0& 1 \end{pmatrix} $$
 with respect to $\underline\C^2=L\oplus\tilde L$ we therefore obtain that $FP^{-1}$ extends holomorphically through the zeros of $\det (F)$ and  $FP^{-1}$ takes values in $\SL(2,\C)$ by construction.
\end{proof}

\begin{definition}
On the loop group $\Lambda \SL(2, \C)$ 
we define an involution $\gamma \mapsto \gamma^*$ by \[\gamma^*(\lambda)=\overline{\gamma(-\bar\lambda^{-1})}.\] 
\end{definition}
\begin{lemma}\label{realrfun}
Let $g$ be as in Lemma \ref{normal-form-O1} with the additional property that
\[gg^*=-\text{Id}.\]
Then, there exists a factorisation of the form \eqref{ghdr}  for a real valued function $r$.
\end{lemma}
\begin{proof}
By Lemma \ref{normal-form-O1}  there exist a factorisation of $g$ with
\[g_y=h_y^+\begin{pmatrix} \lambda^{-1}& r(y)\\0 & \lambda\end{pmatrix} h_y^-\]
for a complex valued function $r$.  We want to find new $h^\pm$ such that the function $r$ is real valued.
Let 
$A:=h^- (h^+)^*.$ Then $A$ is a negative loop and 
 \[-\text{Id}=gg^*=h^+\begin{pmatrix} \lambda^{-1}& r\\  0&\lambda\end{pmatrix}  h^- (h^+)^*\begin{pmatrix} -\lambda& \bar r\\  0&-\lambda^{-1}\end{pmatrix} (h^-)^*\]
implies
\begin{equation}\label{Aeq}(A^{-1})^*=\begin{pmatrix} \lambda^{-1}& r\\  0&\lambda\end{pmatrix} A \begin{pmatrix} \lambda& -\bar r\\  0&\lambda^{-1}\end{pmatrix}.\end{equation}
Comparing the $\lambda$-coefficients of both sides of the equation shows that $A$ is
of the form
\begin{equation}\label{AFORM}
A=\begin{pmatrix} a& b\\ c& d\end{pmatrix}+\lambda^{-1}\begin{pmatrix} a_1& 0\\ c_1& d_1\end{pmatrix}+\lambda^{-2}\begin{pmatrix} 0& 0\\ c_2& 0\end{pmatrix}\end{equation}
for some functions $a,..,c_2\colon \tilde U\to\C$ satisfying
\begin{equation}\label{algconsabc}
\begin{split}
c_2&=-\bar c\\
c_1&=\bar c_1\\
a_1&=\bar c r\\
 d_1&=-\bar c \bar r\\
a+c_1 r&=\bar d\\
\bar b&=c r\bar r\\
-d r+a\bar r+c_1 r\bar r&=0.\\
\end{split}
\end{equation}
The fifth and the last equation in \eqref{algconsabc} gives that the function \[y\mapsto\; d(y)r(y)\] is  real valued.
Moreover, $b= c r \bar r$ gives
\[ a d=1+b c=1+c\bar  c r\bar r,\]
implying that $a$ and $d$ are non-vanishing.
Therefore, 

$$\tilde h^+ =h^+\begin{pmatrix}\tfrac{1}{\sqrt{d}}&0\\0 & \sqrt{d}\end{pmatrix}$$
$$\tilde h^- =\begin{pmatrix} \sqrt{d}& 0\\0 &\tfrac{1}{\sqrt{d}}\end{pmatrix} h^-$$
are well defined and satisfy
\[g_y=\tilde h_y^+\begin{pmatrix} \lambda^{-1}& \tilde r\\0 & \lambda\end{pmatrix} \tilde h_y^- \]
with the real-valued function
\[\tilde r=dr.\]
\end{proof}
\begin{proof}[Proof of Theorem \ref{higherharmonic}]
Consider on $M_q$ 
the family of flat connections $\nabla^\lambda$ and
the associated family of $\SL(2, \C)$-gauge transformations $g(\lambda)$ 
satisfying \eqref{realeqsecsign}.
The set $U \subset M_q$ of points $y$ where the loop $\lambda\mapsto g_y(\lambda)$ lies in the big cell is by Proposition \ref{jump-up-1} open and dense. In other words, for every $y \in U$ there exists a Birkhoff factorisation
\begin{equation}\label{birkhoff-split}g_y(\lambda)=g_y^+(\lambda)g_y^-(\lambda),\end{equation}
where $g_y^+(\lambda)$ and $g_y^-(\lambda)$ are positive and negative loops,  respectively.
We can choose the factorisation \eqref{birkhoff-split} in a way that both factors depend smoothly on $y\in U,$ e.g., by imposing that
$g_y^+(0)=\text{Id}.$

Recall that the Birkhoff factorisation $g=g^+g^-$ is unique up to 
$$g^+\mapsto g^+B^{-1}\quad\text{and} \quad  g^-\mapsto Bg^-$$
for some $B\in\SL(2,\C).$ Hence, by
using \eqref{realeqsecsign}, there exists a smooth map \[B\colon U\to \SL(2,\C)\]
such that
\[g^-(\lambda)=B\overline{g^+(-\bar\lambda^{-1})^{-1}}\]
and \[B\bar B=-\text{Id}.\]
The last equation implies that
\[B=\begin{pmatrix} \alpha &\beta\\\gamma&-\bar\alpha\end{pmatrix}\]
for real-valued $\beta,\gamma,$ and complex-valued $\alpha.$
Because the determinant of $B$ is $1$, 
\[\beta\gamma<0.\]
Depending on the sign of $\beta$, $B$ can be written as 
\[B=\pm C\delta\bar C^{-1}\]
for the smooth map 
\[C=\begin{pmatrix} \sqrt{\pm \beta} &0\\ \mp \frac{\bar \alpha}{\sqrt{\pm \beta}} &\frac{1}{\sqrt{\pm \beta}} \end{pmatrix}\colon U\to\SL(2,\C)\]
and  \[\delta=\begin{pmatrix}0&1\\-1&0\end{pmatrix}.\]

It can be checked directly that the family of flat connections 
\begin{equation}\label{birkhoffgauge}
\tilde\nabla^\lambda=\nabla^\lambda.(g^+(\lambda) C)\end{equation}
 gives a solution of the self-duality equations on 
$U$ with respect to the standard hermitian metric on $\C^2.$

We show that for $p\to\partial U$ the operator norm of the associated harmonic map $H$ into the hyperbolic 3-space goes to $\infty$, and that 
$H$ extends (after glueing the two hyperbolic balls along its $\infty$-boundary $S^2$) to a smooth map $f\colon\widetilde M_q\to S^3$ on the universal covering $\widetilde M_q\to M_q$. 

Recall that on a simply connected subset $\tilde U \subset U$ the  harmonic map $H$ corresponding to the real section $s$ of $\mathcal M_{DH}$  is given by
\[\pm\bar{F}^TF\colon \tilde U \longrightarrow \mathbb H^3\cup \mathbb H^3,\]
where $F$ is the parallel Frame of $\tilde\nabla^1$ in \eqref{birkhoffgauge}.  
A parallel frame $\Psi$ of $\nabla^1$ is then given by 
\[\Psi=(g^+(1) C) F,\]
with $g^+(\lambda)$ defined in \eqref{birkhoff-split}. Therefore, we have to analyse the behaviour of
\begin{equation}\label{fequation}
\begin{split}
f&=\pm \bar F^T F = \pm\bar\Psi^T  \left (\bar g^+(1)^{-1}\right)^T \left(\bar{C}^{-1}\right )^T C^{-1} g^+(1)^{-1}\Psi\\
 &=\bar\Psi^T  \left(\bar g^+(1)^{-1}\right)^T \delta^{-1}(\pm \bar{C} \delta C^{-1}) g^+(1)^{-1}\Psi
 =\bar\Psi^T  \left(\bar g^+(1)^{-1}\right)^T \delta^{-1} \bar B g^+(1)^{-1}\Psi
 \end{split}
 \end{equation}
when $g(\lambda)$ leaves the big cell. Let $x\in W$. By Lemma \ref{normal-form-O1} there exists a neighbourhood $\tilde U\subset M_q$ of $x$ such that for all $y \in \tilde U$ we have the factorisation
\[g_y=h_y^+\begin{pmatrix} \lambda^{-1}& r(y)\\0 & \lambda\end{pmatrix} h_y^-\]
for a $\mathcal C^k$-function $r\colon \tilde U\to \R$ (Lemma \ref{realrfun}) with
sufficiently large $k\in\mathbb N$. Recall that $r\neq 0$ is equivalent to $g_y$ lying in the big cell, since
\[g_y=h^+(y)\begin{pmatrix} r(y)& 0\\  \lambda&\tfrac{1}{r(y)}\end{pmatrix}\begin{pmatrix} \tfrac{1}{r(y)}\lambda^{-1}& 1\\  -1&0\end{pmatrix} h^-(y)\]
with 
$$g^+_y = h^+(y)\begin{pmatrix} r(y)& 0\\  \lambda&\tfrac{1}{r(y)}\end{pmatrix}$$
gives a Birkhoff factorisation of $g_y$.
For $y \in \tilde U$ lying in the big cell \eqref{fequation} thus yields
\begin{equation}\label{fequation2}
\begin{split}
f&=\bar\Psi^T  (\bar g^+(1)^{-1})^T \delta^{-1} \bar B g^+(1)^{-1}\Psi\\
&=\bar\Psi^T  (\bar h^+(1)^{-1})^T \delta^{-1}\begin{pmatrix} \bar r& 0\\  1&\tfrac{1}{\bar r}\end{pmatrix}\bar B\begin{pmatrix} \tfrac{1}{r}& 0\\  -1&r\end{pmatrix} h^+(1)^{-1}\Psi.\\
\end{split}
\end{equation}
For every $y\in \tilde U$ fixed, the map $y\mapsto h^+(1)(y)^{-1}\Psi(y)\colon \tilde U\to\SL(2,\C)$ is well-defined and acts on 
$\mathbb H^3\cup\mathbb H^3$ by isometries via \eqref{SLonS3}. Therefore, to analyse the behaviour for $H \rightarrow \infty,$ it is sufficient to consider the term \begin{equation}\label{fequation3}
\begin{split}
\delta^{-1}\begin{pmatrix} \bar r& 0\\  1&\tfrac{1}{\bar r}\end{pmatrix}\bar B\begin{pmatrix} \tfrac{1}{r}& 0\\  -1&r\end{pmatrix}\\
\end{split}
\end{equation}
for $r(y)\to0.$

As in the proof of Lemma \ref{realrfun} we consider 
$A=h^- (h^+)^*$ of the form \eqref{AFORM} with coefficients satisfying \eqref{algconsabc}.
Then the constant loop
$B$ is given by
\begin{equation}\label{bident}B=g^- (g^+)^*=\begin{pmatrix} \tfrac{1}{r}\lambda^{-1}& 1\\  -1&0\end{pmatrix} A\begin{pmatrix} \bar r& 0\\  -\lambda^{-1}&\tfrac{1}{\bar r}\end{pmatrix}=\begin{pmatrix} c\bar r & \tfrac{d}{\bar r}\\ -a \bar r& -\tfrac{b}{\bar r}  \end{pmatrix}
\end{equation}
and
\begin{equation}\label{fequation4}
\begin{split}
\delta^{-1}\begin{pmatrix} \bar r& 0\\  1&\tfrac{1}{\bar r}\end{pmatrix}\bar B\begin{pmatrix} \tfrac{1}{r}& 0\\  -1&r\end{pmatrix}=
\begin{pmatrix} -c +\tfrac{\bar a}{\bar r}+\tfrac{\bar d}{r}-\bar c& -\bar d+cr\\ (-\tfrac{\bar d}{r}+\bar c) \bar r& \bar d \bar r\end{pmatrix},\\
\end{split}
\end{equation}
where we have used $\bar b=c r\bar r.$
Therefore, the map in \eqref{fequation3} takes values in the space of symmetric matrices $\{C\in\SL(2,\C)\mid C\mapsto \bar C^T\}$ by  \eqref{algconsabc}.  
Because
\[1=\det{\begin{pmatrix} a& b\\ c& d\end{pmatrix}}=ad-bc=d\,\bar d-c_1\, d\, r-c\,\bar c\,r\,\bar r\]
 $d(y)\, \bar d(y)\to 1$ for $r(y)\to 0$. Therefore, 
the upper left entry of the right hand side of \eqref{fequation4} goes to $\infty$ with the same order as the vanishing order of $r$.
The other entries remain finite. 

It remains to prove that the vanishing order of $r$ at $x \in W$ is 1 and that $\partial U$ is a smooth 1-dimensional submanifold for $\rho \sim 0$. Let $\rho = 0$ and $x \in W$.  In this case it is well-known from \cite[$\S$6]{BaBo} that the surface $f$ intersects the boundary at infinity transversely. It follows from \eqref{fequation4}  that the differential of $r$ at $x$ does not vanish.

This property is preserved for $\rho \sim 0$, because the data depend continuously on $\rho$. 
In particular, 
the upper left entry of the right hand side of \eqref{fequation4} goes to $\infty$ with order 1,
while the other entries remain finite valued. 
Lemma \ref{normal-form-O1} shows that $\partial U$ is 
given by the vanishing locus of $r$, thus $\partial U$ is a smooth 1-dimensional submanifold of $M_q$.

By construction, the Higgs fields of the real sections in Theorem \ref{MainTheorem} are nilpotent.
This implies that the corresponding harmonic maps on $\widetilde M$ are conformal. Therefore,  
 $\tilde f_{\mid \widetilde U}$ is a conformally parametrized equivariant minimal surface in the hyperbolic 3-space defined on the preimage  $\widetilde U=\pi^{-1} (U)$. In particular,  $\tilde f_{\mid \widetilde U}$ is a Willmore surface, and as $\tilde f$ is  smooth
and $\widetilde U\subset\widetilde M$ is a open dense subset, $\tilde f$ satisfies the Willmore Euler-Lagrange equation globally.

\end{proof}

\begin{remark}[Higher solutions and Willmore surfaces]\label{remwim}
The higher solutions of Hitchin's self-duality equations constructed here are given by isothermic Willmore surfaces that are 
locally but not globally minimal in a space form. Willmore tori in the 4-sphere are shown to form an integrable system in 
\cite{Bohle}. They  are obtained through an associated family of flat $\SL(4, \C)$-connections $\nabla^\mu$ for $\mu \in \C^*$ \cite{FLPP}. The associated family is encoded by the spectral curve $\Sigma$ which is a $4$-fold covering of $\C P^1$ and 
possesses an additional involution $\sigma$ if the target is a $3$-dimensional space form. The quotient $\Sigma/\sigma$ is 
then a hyperelliptic curve. In the case of isothermic Willmore tori this quotient is another $\C P^1$ and the family of flat $\SL(4,\C)$-connections splits into the direct sum of two (gauge equivalent) rank 2 families of flat connections parametrized 
by $\lambda \in \Sigma/\sigma $, see \cite{Hel2}.  The double covering of the $\mu$-plane by the $\lambda$-plane 
corresponds to taking a square root. Therefore, the rank 2 associated family of flat connections obtained through this 
construction is invariant under a real involution covering either  $\lambda\mapsto\bar\lambda^{-1}$ or $\lambda\mapsto-\bar\lambda^{-1}$ on $\C P^1$, i.e., it corresponds to the harmonic maps into the 3-sphere or the self-duality equations case. We expect 
more sophisticated real sections to emerge from Willmore tori where $\Sigma/\sigma$ has non-trivial topology. 
 \end{remark}

\begin{acknowledgements}
The second author was supported by RTG 1670 {\em Mathematics inspired by string theory and quantum field theory} funded by the Deutsche Forschungsgemeinschaft (DFG).
\end{acknowledgements}

\end{document}